\documentclass[a4paper]{amsart}
\usepackage{graphicx}
\usepackage{amssymb}
\usepackage{amsmath}
\usepackage{amsthm}
\usepackage{amscd}
\usepackage{color}
\usepackage{colordvi}
\usepackage[all,2cell]{xy}
\usepackage{picture}
\usepackage{multirow}
\usepackage{makecell}
\usepackage{dsfont}
\usepackage{extarrows}
\usepackage{float}
\usepackage{mathtools}
\usepackage{CJKutf8}

\usepackage{amsfonts,enumerate,verbatim,mathtools,tikz,bm,tikz-cd,hyperref,comment}
\hypersetup{
	colorlinks=true,
	linkcolor=blue,
	citecolor=red,
	filecolor=magenta,
	urlcolor=cyan,
}

\UseAllTwocells \SilentMatrices
\newtheorem{thm}{Theorem}[section]

\newtheorem{cor}[thm]{Corollary}
\newtheorem{lem}[thm]{Lemma}
\newtheorem{exm}[thm]{Example}

\newtheorem{pro}[thm]{Proposition}
\theoremstyle{definition}
\newtheorem{dfn}[thm]{Definition}
\theoremstyle{remark}
\newtheorem{rem}[thm]{\bf Remark}
\numberwithin{equation}{section}

\DeclareMathOperator{\projdim}{\mathrm{proj.dim}}
\DeclareMathOperator{\injdim}{\mathrm{inj.dim}}

\DeclareMathOperator{\Ker}{\mathrm{Ker}}

\begin{document}
	\begin{CJK}{UTF8}{gbsn}
	\title[Cotorsion pairs in extensions of abelian categories ]{Cotorsion pairs in extensions of abelian categories}

	\author{Dongdong Hu}

\makeatletter
\@namedef{subjclassname@2020}{\textup{2020} Mathematics Subject Classification}
\makeatother

\date{\today}

\thanks{E-mail: hudd@mail.ustc.edu.cn}

\maketitle

	\begin{abstract}
		Let $\mathcal{B}$ be an abelian category with enough projective objects and enough injective objects and let $\mathcal{A}=\mathcal{B}\ltimes_\eta\mathsf{F}$ be an $\eta$-extension of $\mathcal{B}$. Given a cotorsion pair $(\mathcal{X},\;\mathcal{Y})$ in $\mathcal{B}$, we construct a cotorsion pair $(\prescript{\perp}{}{\mathsf{U}^{-1}(\mathcal{Y})},\;\mathsf{U}^{-1}(\mathcal{Y}))$ in $\mathcal{A}$ and a cotorsion pair $(\Delta(\mathcal{X}),\;\Delta(\mathcal{X})^\perp)$ in $\mathcal{A}$ for $\mathsf{F}^2=0$. In addition, the heredity and completeness of these cotorsion pairs are studied. Finally, we give some applications and examples in comma categories, some Morita context rings and trivial extensions of rings.
		
	\end{abstract}
	\maketitle
	\section{Introduction}
	
	Cotorsion pairs, which were introduced in \cite{D}, are crucial in relative homological algebra and model structures; see \cite{EJ,GJ,H}.  A cotorsion pair consists of two classes of objects that are orthogonal with respect to the Ext functor.
	
	 Cotorsion pairs in different fields have been the focus of intense research. Cotorsion pairs in the comma categories  are studied in \cite{HW,HZ}. In \cite{CRZ}, the authors investigated cotorsion pairs and model structures on some Morita context rings, while those on the special case of formal triangular matrix rings have been studied in \cite{DPZ,Mao}. For more about the construction of cotorsion pairs, see \cite{DLLX,EPZ,HJ,Mao2,O}.
	
 An $\eta$-extension category $\mathcal{B}\ltimes_\eta\mathsf{F}$ of an abelian category $\mathcal{B}$ by a right exact functor $\mathsf{F}:\mathcal{B}\rightarrow\mathcal{B}$ was introduced in \cite{M}, where $\eta:\mathsf{F}^2\rightarrow\mathsf{F}$ is an associative natural transformation; see Subsection~\ref{subsec2.1}. In addition, there is a dual concept called $\zeta$-coextensions; see Subsection~\ref{subsec2.2} for details. Throughout this paper, we assume that the abelian category $\mathcal{B}$ has enough projective objects and enough injective objects. 	We are concerned with cotorsion pairs in some extensions of abelian categories. More precisely, we describe how to construct cotorsion pairs in an extension category of $\mathcal{B}$ from cotorsion pairs in $\mathcal{B}$. The motivation is the result in \cite{CRZ}, which studied cotorsion pairs and model structures on Morita context rings. We mention other related results in \cite{DPZ,Mao2,HW}. The first result of this paper is a vast generalization of the mentioned results; see Theorem~A.
 
 Let  $\mathcal{A}$ be an abelian category with enough projective objects and enough injective objects.
 For a class $\mathcal{X}$ of objects of $\mathcal{A}$, let
 \[\begin{aligned}&\prescript{\perp}{}{\mathcal{X}}=\{Y\in\mathcal{A}\;|\;\mathrm{Ext}_\mathcal{A}^1(Y,X)=0\mbox{ for all }X\in\mathcal{X}\},\\
 	&\mathcal{X}^\perp=\{Y\in\mathcal{A}\;|\;\mathrm{Ext}_\mathcal{A}^1(X,Y)=0\mbox{ for all }X\in\mathcal{X}\}.
 \end{aligned}\]
 
 Let $\mathcal{B}$ be an abelian category. For a class $\mathcal{X}$ of objects in $\mathcal{B}$, two classes $\mathsf{U}^{-1}(\mathcal{X})$ and $\Delta(\mathcal{X})$ are defined; see Subsection~\ref{3.1}. For an additive functor $\mathsf{F}$ between abelian categories, let $\{\mathbb{L}_n\mathsf{F}\}_{n\in\mathbb{Z}}$ be the left derived functors of $\mathsf{F}$. Constructions of (hereditary) cotorsion pairs in $\mathcal{B}\ltimes_\eta\mathsf{F}$ are given as follows. 
 
 \vskip 5pt
 
 \noindent {\bf Theorem\;A}\;(=\;Theorems~\ref{thm1} and \ref{thm2}).\;\emph{ Let $\mathcal{B}$ be an abelian category, and let $\mathcal{B}\ltimes_\eta\mathsf{F}$ be an $\eta$-extension of $\mathcal{B}$. Assume that $(\mathcal{X},\;\mathcal{Y})$ is a cotorsion pair in $\mathcal{B}$.
 	\begin{enumerate}[(1)]
 		\item If $\mathbb{L}_1\mathsf{F}(\mathcal{X})=0$,  then $(\prescript{\perp}{}{\mathsf{U}^{-1}(\mathcal{Y})},\mathsf{U}^{-1}(\mathcal{Y}))$ is a cotorsion pair in $\mathcal{B}\ltimes_\eta\mathsf{F}$; and moreover, it is hereditary if and only if so is $(\mathcal{X,Y})$. 
 		\item If $\mathsf{F}^2=0$, then $(\Delta(\mathcal{X}),\;\Delta(\mathcal{X})^\perp)$ is a cotorsion pair in $\mathcal{B}\ltimes\mathsf{F}$. Moreover, if $\mathbb{L}_1\mathsf{F}(\mathcal{X})=0$, then $(\Delta(\mathcal{X}),\Delta(\mathcal{X})^\perp)$ is hereditary if and only if so is $(\mathcal{X,Y})$.
 	\end{enumerate}
 }

It is typically difficult to describe  the classes $\prescript{\perp}{}{\mathsf{U}^{-1}(\mathcal{Y})}$  and $\Delta(\mathcal{X})^\perp$. Under certain conditions, we characterize $\prescript{\perp}{}{\mathsf{U}^{-1}(\mathcal{Y})}$  and $\Delta(\mathcal{X})^\perp$ in Theorem~\ref{thm4}. We  emphasize that  $(\prescript{\perp}{}{\mathsf{U}^{-1}(\mathcal{Y})},\mathsf{U}^{-1}(\mathcal{Y}))\ne (\Delta(\mathcal{X}),\Delta(\mathcal{X})^\perp)$ in general; see \cite[Example~4.3]{CRZ}. In some special cases, these pairs may coincide; see Theorems~\ref{thm4}, \ref{thm6} and \ref{thm7} and \cite[Proposition~3.7]{DPZ}.

When the abelian category $\mathcal{B}$ is Frobenius, its trivial extension retains good properties. Applying Theorem~A to the cotorsion pair $(\mathcal{B},\;\prescript{}{\mathcal{B}}{\mathcal{I}})$, we obtain Theorem~\ref{thm6}. Although cotorsion pairs in Theorem~\ref{thm6} are precisely the Gorenstein projective cotorsion pairs, we establish a characterization of Gorenstein projective objects in $\mathcal{B}\ltimes\mathsf{F}$.  We mention that Theorem~\ref{thm6} is a categorical version of \cite[Theorem~4.6]{CRZ}. For more results about Gorenstein projective objects, see \cite{HLXZ,K,LZ,Luo-Z,Z}.

In cotorsion theory, completeness of cotorsion pairs plays a crucial role. This is evident not only in the theory itself, but also in finding model structures; see  \cite{BR,CLZ,H}. The following result demonstrates the completeness of the cotorsion pair in Theorem~A.

\vskip 6pt

\noindent {\bf Theorem\;B}\;(=\;Theorem~\ref{thm5}).\;\emph{ Let $\mathcal{B}\ltimes\mathsf{F}$ be a right trivial extension of an abelian category $\mathcal{B}$. Assume that $(\mathcal{X,Y})$ is a hereditary complete cotorsion pair. If $\mathbb{L}_1\mathsf{F}(\mathcal{X})=0$ and $\mathsf{F}(\mathcal{X})\subseteq\mathcal{Y}$, then $(\prescript{\perp}{}{\mathsf{U}^{-1}(\mathcal{Y})},\mathsf{U}^{-1}(\mathcal{Y}))$ is a hereditary complete cotorsion pair.
}

\vskip 6pt

The key to proving Theorem~B lies in Proposition~\ref{cover} and Lemma~\ref{can-seq}. In particular, for $\mathsf{F}^2=0$, $(\Delta(\mathcal{X}),\Delta(\mathcal{X})^\perp)$ is a cotorsion pair in $\mathcal{B}\ltimes\mathsf{F}$ by Theorem~A. Combining Theorem~\ref{thm4} and Theorem~B, we obtain that $(\Delta(\mathcal{X}),\Delta(\mathcal{X})^\perp)$ is complete; see Corollary~\ref{cor1}. As an application of Theorem~B, we construct a hereditary Hovey triple and a contravariantly finite cotorsion pair in $\mathcal{B}\ltimes\mathsf{F}$; see Theorems~\ref{thm9} and~\ref{thm10}.

We mention that some examples of extensions of abelian categories are summarised in \cite{B}. For example, the module category of a Morita context ring is a right trivial extension of an abelian category. Theorem~A recovers \cite[Theorems~1.1 and~1.2]{CRZ}. Another example is the comma category. As an application, we studied cotorsion pairs and model structures in comma categories; see Section~\ref{sec4} for details.​​

The paper is structured as follows. In Section~\ref{sec2}, we recall some basic concepts. Subsections~\ref{subsec2.1} and \ref{subsec2.2} focus on extensions of abelian categories, detailing their basic notions and properties. Subsequently, we review the theory of cotorsion pairs, model structures, and Gorenstein categories. In Section~\ref{sec3}, we explicitly construct cotorsion pairs in extensions of abelian categories; see Theorems~\ref{thm1} and ~\ref{thm2}. We also study the completeness of the cotorsion pair; see Theorem~\ref{thm5}. Finally, we give some applications and examples in comma categories, some Morita context rings and trivial extensions of rings in Section~\ref{sec4}.

	\section{Preliminaries}\label{sec2}

\subsection{$\eta$-extensions of abelian categories}\label{subsec2.1} We recall from \cite{M} (see also \cite{B}) the construction of the $\eta$-extension of an abelian category. Let $\mathcal{B}$ be an abelian category with enough projective objects and enough injective objects, let $\mathsf{F}:\mathcal{B}\rightarrow\mathcal{B}$ be a covariant, additive, right exact functor and $\eta:\mathsf{F}^2\rightarrow\mathsf{F}$ a natural transformation. Assume that $\eta$ is \emph{associative}, that is, $\eta\circ\mathsf{F}\eta=\eta\circ\eta\mathsf{F}$, i.e. the following diagram commutes:
\[\begin{tikzcd}
	\mathsf{F}^3 \arrow[d, "\eta\mathsf{F}"'] \arrow[r, "\mathsf{F}\eta"] & \mathsf{F}^2 \arrow[d, "\eta"] \\
	\mathsf{F}^2 \arrow[r, "\eta"]                                        & \mathsf{F}                    
\end{tikzcd}\]

 We denote by $\mathcal{B}\ltimes_\eta\mathsf{F}$ the category (called an \emph{$\eta$-extension} of $\mathcal{B}$) defined as follows: The objects are pairs $(X,f)$, where $f\in\mathrm{Hom}_\mathcal{B}(F(X),X)$ such that $f\circ \mathsf{F}(f)=f\circ\eta_X$. A morphism $\alpha:(X,f)\rightarrow(Y,g)$ in $\mathcal{B}\ltimes_\eta\mathsf{F}$ is a morphism $\alpha\in\mathrm{Hom}_\mathcal{B}(X,Y)$ such that $\alpha\circ f=g\circ \mathsf{F}(\alpha)$. One can prove that $\mathcal{B}\ltimes_\eta\mathsf{F}$ is an abelian category; see \cite[Proposition~1.2]{M}. If $\eta=0$, then $\mathcal{B}\ltimes\mathsf{F}:=\mathcal{B}\ltimes_0\mathsf{F}$ is the \emph{right trivial extension} of an abelian category $\mathcal{B}$ by the functor $\mathsf{F}$; see \cite[Section~1]{FPR}.
 
 There are the following functors:
 \[\begin{tikzcd}
 	\mathcal{B} \arrow[r, "\mathsf{Z}"', shift right] & \mathcal{B}\ltimes_\eta\mathsf{F} \arrow[r, "\mathsf{U}"', shift right] \arrow[l, "\mathsf{C}"', shift right] & \mathcal{B} \arrow[l, "\mathsf{T}"', shift right]
 \end{tikzcd}\]
 
 \begin{itemize}
 	\item $\mathsf{U}(X,f)=X$ and $\mathsf{U}(\alpha)=\alpha$;
 	\item $\mathsf{Z}(X)=(X,0)$ and $\mathsf{Z}(\alpha)=\alpha$;
 	\item $\mathsf{T}(X)=\left(X\oplus\mathsf{F}(X),t_X:=\left(\begin{smallmatrix}
 		0&0\\
 		1&\eta_X
 	\end{smallmatrix}\right)\right)$ and $\mathsf{T}(\alpha)=\begin{pmatrix}
 	\alpha&0\\
 	0&\mathsf{F}(\alpha)
 	\end{pmatrix}$;
 	\item $\mathsf{C}(X,f)=\mathrm{Coker}(f)$ and $\mathsf{C}(\alpha)$ is the morphism induced by the cokernel.
 \end{itemize}
 
 The associativity of $\eta$ guarantees that $\mathsf{T}(X)$ is indeed an object of $\mathcal{B}\ltimes_\eta\mathsf{F}$ (the converse is also true). The following proposition describes some homological properties of  $\eta$-extensions; see \cite[Proposition~1.2, Lemma~2.1, Proposition~2.6, and Proposition~2.7]{M}.

\begin{pro}\label{pro1}
	Let $\mathcal{B}\ltimes_\eta\mathsf{F}$ be an $\eta$-extension of $\mathcal{B}$. Then the following hold.
	\begin{enumerate}[(1)]
		\item The sequence $(X,f)\xrightarrow{\alpha}(Y,g)\xrightarrow{\beta}(Z,h)$ is exact in $\mathcal{B}\ltimes_\eta\mathsf{F}$ if and only if the underlying sequence$ X\xrightarrow{\alpha}Y\xrightarrow{\beta}Z$ is exact in $\mathcal{B}$. In particular, the functor $\mathsf{U}$ is exact.
		\item The pairs $(\mathsf{T,U})$ and $(\mathsf{C,Z})$ are adjoint pairs of functors, and $\mathsf{CT}\cong\mathsf{Id}_\mathcal{B}$.
		\item The object $(X,f)\in\mathcal{B}\ltimes_\eta\mathsf{F}$ is projective if and only if $\mathrm{Coker}(f)\in\mathcal{B}$ is projective and $\mathsf{T}(\mathrm{Coker}(f))\cong(X,f)$ in $\mathcal{B}\ltimes_\eta\mathsf{F}$.
		\item If $\mathcal{B}$ has enough projective objects, then so does $\mathcal{B}\ltimes_\eta\mathsf{F}$.
	\end{enumerate}
\end{pro}

\subsection{$\zeta$-coextensions of abelian categories}\label{subsec2.2}

There is a dual notion to $\eta$-extensions which we recall below \cite{M}. Let $\mathsf{G}:\mathcal{B}\rightarrow\mathcal{B}$ be a covariant, additive, left exact functor and $\zeta:\mathsf{G}\rightarrow\mathsf{G}^2$ a natural transformation. Assume that $\zeta$ is \emph{coassociative}, that is, $\mathsf{G}\zeta\circ\zeta=\zeta\mathsf{G}\circ\zeta$, i.e. the following diagram commutes:
\[\begin{tikzcd}
	\mathsf{G} \arrow[d, "\zeta"'] \arrow[r, "\zeta"] & \mathsf{G}^2 \arrow[d, "\mathsf{G}\zeta"] \\
	\mathsf{G}^2 \arrow[r, "\zeta\mathsf{G}"]                                        & \mathsf{G}^3                    
\end{tikzcd}\]

 We denote by $\mathsf{G}\rtimes_\zeta\mathcal{B}$ the category (called a \emph{$\zeta$-coextension} of $\mathcal{B}$) defined as follows: The objects are pairs $[Y,g]$, where $g\in\mathrm{Hom}_\mathcal{B}(X,\mathsf{G}(X))$ such that $\mathsf{G}(g)\circ g=\zeta_X\circ g$. A morphism $\alpha:[X,f]\rightarrow[Y,g]$ in $\mathsf{G}\rtimes\mathcal{B}$ is a morphism $\alpha\in\mathrm{Hom}_\mathcal{B}(X,Y)$ such that $\mathsf{G}(\alpha)\circ f=g\circ \alpha$. One can prove that $\mathsf{G}\rtimes_\zeta\mathcal{B}$ is an abelian category; see \cite[Proposition~1.2]{M}. If $\zeta=0$, then $\mathsf{G}\rtimes\mathcal{B}:=\mathsf{G}\rtimes_0\mathcal{B}$ is the \emph{left trivial extension} of an abelian category $\mathcal{B}$ by the functor $\mathsf{G}$; see \cite[Section~1]{FPR}.

 There are the following functors:
\[\begin{tikzcd}
	\mathcal{B} \arrow[r, "\mathsf{Z}"', shift right] & \mathsf{G}\rtimes_\zeta\mathcal{B} \arrow[r, "\mathsf{U}"', shift right] \arrow[l, "\mathsf{K}"', shift right] & \mathcal{B} \arrow[l, "\mathsf{H}"', shift right]
\end{tikzcd}\]

\begin{itemize}
	\item $\mathsf{U}[X,f]=X$ and $\mathsf{U}(\alpha)=\alpha$;
	\item $\mathsf{Z}(X)=[X,0]$ and $\mathsf{Z}(\alpha)=\alpha$;
	\item $\mathsf{H}(X)=\left[\mathsf{G}(X)\oplus X,s_X:=\left(\begin{smallmatrix}
		\zeta_X&0\\
		1&0
	\end{smallmatrix}\right)\right]$ and $\mathsf{H}(\alpha)=\begin{pmatrix}
		\alpha&0\\
		0&\mathsf{G}(\alpha)
	\end{pmatrix}$;
	\item $\mathsf{K}[X,f]=\mathrm{Ker}(f)$ and $\mathsf{K}(\alpha)$ is the morphism induced by the kernel.
\end{itemize}

The coassociativity of $\zeta$ guarantees that $\mathsf{H}(X)$ is indeed an object of $\mathsf{G}\rtimes_\zeta\mathcal{B}$ (the converse is also true). The following proposition describes some homological properties of  $\zeta$-coextensions; see \cite[Proposition~1.2, Lemma~3.1, Proposition~3.2]{M}.

\begin{pro}\label{pro2}
Let $\mathsf{G}\rtimes_\zeta\mathcal{B}$ be the $\zeta$-coextension of an abelian category $\mathcal{B}$. Then the following statements hold.
	\begin{enumerate}[(1)]
		\item The sequence $[X,f]\xrightarrow{\alpha}[Y,g]\xrightarrow{\beta}[Z,h]$ is exact in $\mathsf{G}\rtimes_\zeta\mathcal{B}$ if and only if the underlying sequence $ X\xrightarrow{\alpha}Y\xrightarrow{\beta}Z$ is exact in $\mathcal{B}$. In particular, the functor $\mathsf{U}$ is exact.
		\item The pairs $(\mathsf{H,U})$ and $(\mathsf{K,Z})$ are adjoint pairs of functors, and $\mathsf{KH}\cong\mathsf{Id}_\mathcal{B}$.
		\item The object $[X,f]\in\mathsf{G}\rtimes_\zeta\mathcal{B}$ is injective if and only if $\mathrm{Ker}(f)\in\mathcal{B}$ is injective and $\mathsf{H}(\mathrm{Ker}(f))\cong[X,f]$ in $\mathsf{G}\rtimes_\zeta\mathcal{B}$.
		\item If $\mathcal{B}$ has enough injective objects, then so does $\mathsf{G}\rtimes_\zeta\mathcal{B}$.
	\end{enumerate}
\end{pro}

Let $\mathsf{F}:\mathcal{B}\rightarrow\mathcal{B}$ be a covariant, additive, right exact functor and $\eta:\mathsf{F}^2\rightarrow\mathsf{F}$ a natural transformation. Assume that $\mathsf{F}$ admits a right adjoint functor $\mathsf{G}$. Let $\varphi_{X,Y}$ be the adjunction isomorphism from $\mathrm{Hom}_\mathcal{B}(\mathsf{F}X,Y)$ to $\mathrm{Hom}_\mathcal{B}(X,\mathsf{G}Y)$ for $X,Y\in\mathcal{B}$ and $\psi_{X,Y}=\varphi_{X,Y}^{-1}$. For any $X\in\mathcal{B}$, let $\zeta_X$ denote the composed morphism
\[\mathsf{G}(X)\xrightarrow{\varphi^2(\eta_{\mathsf{G}X})}\mathsf{G}^2\mathsf{FG}X\xlongrightarrow{\mathsf{G}^2(\psi_{\mathsf{G}X,X}(1_{\mathsf{G}X}))}\mathsf{G}^2X.\]
Here, $\varphi^2$ means $\varphi_{\mathsf{G}X,\mathsf{GFG}X}\circ\varphi_{\mathsf{FG}X,\mathsf{FG}X}$. Then there is the following result; see \cite[Section~4]{M}.

\begin{pro}\label{pro3}  Keep the notation as above. Then the following statements hold.
	\begin{enumerate}[(1)]
		\item $\zeta:\mathsf{G}\rightarrow\mathsf{G}^2$ is a natural transformation.
		\item If $\eta=0$, then $\zeta=0.$
		\item If $\eta$ is associative, then $\zeta$ is coassociative.
		\item The functor $\mathsf{\Phi}:\mathcal{B}\ltimes_\eta\mathsf{F}\rightarrow\mathsf{G}\rtimes_\zeta\mathcal{B},\;(X,f)\mapsto[X,\varphi_{X,X}(f)]$ is an isomorphism of categories.
	\end{enumerate}
\end{pro}

\subsection{Cortorsion pairs}\label{subsec2.3}

Let $\mathcal{A}$ be an abelian category with enough projective objects and enough injective objects. A pair $(\mathcal{C,F})$ of classes of objects of $\mathcal{A}$ is a \emph{cotorsion pair} \cite{GJ,S}, if $\mathcal{C}=\prescript{\perp}{}{\mathcal{F}}$ and $\mathcal{F}=\mathcal{C}^\perp$. We said that a cotorsion pair $(\mathcal{X,Y})$ is \emph{contravariantly finite}, if $\omega:=\mathcal{X}\cap\mathcal{Y}$ is contravariantly finite, that is, for any object $A\in\mathcal{A}$, there exists a morphism $f_A:X_A\rightarrow A$ with $X_A\in\omega$ such that every morphism $X\rightarrow A$ with $X\in\omega$ factors through $f_A$.

A cotorsion pair $(\mathcal{C,F})$ is \emph{complete}, if for any object $X\in\mathcal{A}$, there are exact sequences
\[0\longrightarrow F\longrightarrow C\longrightarrow X\longrightarrow0\;\mbox{  and  }\;\;0\longrightarrow X\longrightarrow F'\longrightarrow C'\longrightarrow0,\]
with $C,C'\in\mathcal{C}$ and $F,F'\in\mathcal{F}$.

\begin{pro}$($\cite[Proposition~7.1.7]{EJ}$)$
	Let $\mathcal{A}$ be an abelian category with enough projective objects and enough injective objects. Assume that $(\mathcal{C,F})$ is a cotorsion pair in $\mathcal{A}$. 
	Then the following conditions are equivalent:
	\begin{enumerate}
		\item $(\mathcal{C,F})$ is complete;
		\item For any object $X\in\mathcal{A}$, there is an exact sequence
		$0\rightarrow F\rightarrow C\rightarrow X\rightarrow0$
		with $C\in\mathcal{C}$ and $F\in\mathcal{F}$;
		\item For any object $X\in\mathcal{A}$, there is an exact sequence $0\rightarrow X\rightarrow F'\rightarrow C'\rightarrow0$
		with $C'\in\mathcal{C}$ and $F'\in\mathcal{F}$.
	\end{enumerate}
\end{pro}

A cotorsion pair $(\mathcal{C,F})$ is \emph{hereditary}, provided that  it satisfies one of the equivalent conditions in the following proposition; see \cite[Theorem~1.2.10]{G}, \cite[Lemma~2.2.10]{GJ} and \cite[Proposition~2.5]{CRZ}.

\begin{pro}
	Let $\mathcal{A}$ be an abelian category with enough projective objects and enough injective objects. Assume that $(\mathcal{C,F})$ is a cotorsion pair in $\mathcal{A}$. Then the following conditions are equivalent:
	\begin{enumerate}[(1)]
		\item $\mathrm{Ext}_\mathcal{A}^2(\mathcal{C,F})=0$;
		\item $\mathrm{Ext}_\mathcal{A}^i(\mathcal{C,F})=0$ for $i\ge 1$;
		\item $\mathcal{C}$ is closed under the kernels of epimorphisms;
		\item $\mathcal{F}$ is closed under the cokernels of monomorphisms.
	\end{enumerate}
\end{pro}
\subsection{Approximations}

For an abelian category $\mathcal{A}$ with enough projective objects and enough injective objects, let $\mathcal{X}$ be a class of objects of $\mathcal{A}$, and let $A\in \mathcal{A}$. An epimorphism $f:X\rightarrow A$ with $X\in\mathcal{X}$ is called a \emph{special right $\mathcal{X}$-approximation} of $A$, if $\Ker(f)\in\mathcal{X}^\perp$. A monomorphism $f:A\rightarrow X$ with $X\in\mathcal{X}$ is called a \emph{special left $\mathcal{X}$-approximation} of $A$, if $\mathrm{Coker}(f)\in\prescript{\perp}{}{\mathcal{X}}$; see \cite[Definition~2.7]{BR}.

 \begin{pro}\label{cover}$($\cite[Theorem~3.1]{AA}$)$
 	Let $\mathcal{A}$ be an abelian category with enough projective objects and enough injective objects, and let $0\rightarrow A\rightarrow B\rightarrow C\rightarrow0$ be an exact sequence in $\mathcal{A}$. Assume that $(\mathcal{X,Y})$ is a hereditary cotorsion pair in $\mathcal{A}$.
 	\begin{enumerate}[(1)]
 		\item If $A$ and $C$ have special right $\mathcal{X}$-approximations, then $B$ has a special right $\mathcal{X}$-approximation.
 		\item If $A$ and $C$ have special left $\mathcal{Y}$-approximations, then $B$ has a special left $\mathcal{Y}$-approximation.
 	\end{enumerate}
 \end{pro}
 
 \subsection{Gorenstein categories}\label{subsec2.5} 
 For an abelian category $\mathcal{A}$ with enough projective objects and enough injective objects, let $\prescript{}{\mathcal{A}}{\mathcal{P}}$ (respectively, $\prescript{}{\mathcal{A}}{\mathcal{I}}$)  be the full subcategory of $\mathcal{A}$ consisting of projective (respectively, injective) objects. Denote by $\prescript{}{\mathcal{A}}{\mathcal{P}}^{<\infty}$ (respectively, $\prescript{}{\mathcal{A}}{\mathcal{I}}^{<\infty}$) the full subcategory of $\mathcal{A}$ consisting of objects with finite projective (respectively, injective) dimension. Let
 \[\prescript{}{\mathcal{A}}{\mathcal{P}}^{\le 1}=\{X\in\mathcal{A}\;|\;\projdim X\le 1\}\;\text{ and }\;\prescript{}{\mathcal{A}}{\mathcal{I}}^{\le 1}=\{X\in\mathcal{A}\;|\;\injdim X\le 1\}.\]
 
 We recall from \cite{BR} that the \emph{finitistic projective dimension} $\mathrm{FPD}(\mathcal{A})$ of $\mathcal{A}$ is defined by
 \[\mathrm{FPD}(\mathcal{A}):=\sup\{\projdim X\;|\;X\in\prescript{}{\mathcal{A}}{\mathcal{P}}^{<\infty}\},\]
 and that the \emph{finitistic injective dimension} $\mathrm{FID}(\mathcal{A})$ of $\mathcal{A}$ is defined by
 \[\mathrm{FID}(\mathcal{A}):=\sup\{\injdim X\;|\;X\in\prescript{}{\mathcal{A}}{\mathcal{I}}^{<\infty}\}.\]
 Consider the following dimensions of $\mathcal{A}$:
 \[\mathrm{spli}(\mathcal{A}):=\sup\{\projdim I\;|\; I\in\prescript{}{\mathcal{A}}{\mathcal{I}}\}\mbox{ and }\mathrm{silp}(\mathcal{A}):=\sup\{\injdim P\;|\;P\in\prescript{}{\mathcal{A}}{\mathcal{P}}\}.\]
 The following definition generalises the notion of an Iwanaga-Gorenstein ring.
 \begin{dfn}$($\cite[Definition~VIII~2.1]{BR}$)$
 	An abelian category $\mathcal{A}$ with enough projective objects and enough injective objects is called \emph{Gorenstein}, if $\mathrm{spli}(\mathcal{A})<\infty$ and $\mathrm{silp}(\mathcal{A})<\infty$.
 \end{dfn}
 
 A complex $P^{\bullet}$ of projective objects of $\mathcal{A}$ is called \emph{totally acyclic}, if it is acyclic and $\mathrm{Hom}_\mathcal{A}(P^\bullet,Q)$ is acyclic for any projective objective $Q$ of $\mathcal{A}$. An object $X\in\mathcal{A}$ is called \emph{Gorenstein projective}, if there is a totally acyclic complex $P^\bullet$ of projective objects of $\mathcal{A}$ with $X=\mathrm{Im}(P^{-1}\rightarrow P^0)$. We denote by $\mathrm{GP}(\mathcal{A})$ the full subcategory of $\mathcal{A}$ consisting of Gorenstein projective objects.
 \begin{thm}\label{Gor}$($\cite[Theorem~VIII~2.2]{BR}$)$
 	Let $\mathcal{A}$ be a Gorenstein category. Then the following statements hold.
 	\begin{enumerate}[(1)]
 	\item We have $\prescript{}{\mathcal{A}}{\mathcal{P}}^{<\infty}=\prescript{}{\mathcal{A}}{\mathcal{I}}^{<\infty}$.
 	\item We have the equalities $\mathrm{FPD}(\mathcal{A})=\mathrm{spli}(\mathcal{A})=\mathrm{silp}(\mathcal{A})=\mathrm{FID}(\mathcal{A})$.
 	\item $(\mathrm{GP}(\mathcal{A}),\prescript{}{\mathcal{A}}{\mathcal{P}}^{<\infty})$ is a  hereditary complete cotorsion pair.
 	\end{enumerate}
 \end{thm}

\subsection{Hovey triples}	\label{subsec2.6}

Let $\mathcal{A}$ be an abelian category. We recall from \cite{H} that a triple $(\mathcal{C,F,W})$ of classes of objects of $\mathcal{A}$ is a \emph{$($hereditary$)$ Hovey triple} in $\mathcal{A}$, if it satisfies the following conditions:
\begin{enumerate}[(1)]
	\item Both $(\mathcal{C}\cap\mathcal{W},\mathcal{F})$ and $(\mathcal{C},\mathcal{F}\cap\mathcal{W})$ are (hereditary) complete cotorsion pairs;
	\item The class $\mathcal{W}$ is \emph{thick}, that is, $\mathcal{W}$ is closed under direct summands, and if two out of three terms in a short exact sequence are in $\mathcal{W}$, then so is the third.
\end{enumerate}

\begin{thm}\label{corr}$($Hovey correspondence,\;\cite[Theorem~2.2]{H}$)$
	Let $\mathcal{A}$ be an abelian category. Then there is a one-to-one correspondence Hovey triples with abelian model structures in the sense of \cite[Definition~2]{Gi} in $\mathcal{A}$, given by
	$$(\mathcal{C,\;F,\;W})\longmapsto(\mathrm{Cofib}(\mathcal{A}),\;\mathrm{Fib}(\mathcal{A}),\;\mathrm{Weq}(\mathcal{A})),$$
	where
	\begin{equation*}
		\begin{aligned}
			&\mathrm{Cofib}(\mathcal{A})=\{f\;|\;f\mbox{ is a monomorphism with }\mathrm{Coker}(f)\in\mathcal{C}\},\\ 
				&\mathrm{Fib}(\mathcal{A})=\{g\;|\;g\mbox{ is an epimorphism with }\mathrm{Ker}(g)\in\mathcal{F}\},\\
			&\mathrm{Weq}(\mathcal{A})=\left\{p\circ i\;\left|\;\begin{minipage}{0.5\textwidth}
				$i$ is monic with $\mathrm{Coker}(i)\in\mathcal{C}\cap\mathcal{W}$, and
				$p$ is epic with $\mathrm{Ker}(p)\in\mathcal{F}\cap\mathcal{W}$.
			\end{minipage}\right\}\right.
		\end{aligned}
	\end{equation*}
\end{thm}
\vskip 5pt
Let $\mathcal{M}=(\mathcal{C,F,W})$ be a hereditary Hovey triple in $\mathcal{A}$ and let $W=\mathrm{Weq}(\mathcal{A})$ be given by Theorem~\ref{corr}. The Quillen's homotopy category of $\mathcal{A}$ is the localization $\mathcal{A}[W^{-1}]$, denoted by $\mathrm{Ho}(\mathcal{M})$. For the detail, we refer to \cite[Section~2]{Gi}. We know that if  $\mathcal{M}=(\mathcal{C,F,W})$ is a hereditary Hovey triple in $\mathcal{A}$, then $\mathrm{Ho}(\mathcal{M})$ is a triangulated category; see \cite[Subsection~4.1]{Gi}.

\begin{thm}\label{te}
	Let $\mathcal{A}$ be an abelian category. Assume that $\mathcal{M}=(\mathcal{C,F,W})$ is a hereditary Hovey triple in $\mathcal{A}$. Then the following hold:
	\begin{enumerate}[(1)]
	\item $\mathcal{C}\cap\mathcal{F}$ is a Frobenius category with the canonical exact structure, with $\mathcal{C}\cap\mathcal{F}\cap\mathcal{W}$ as the class of projective-injective objects;
	\item The composition $\mathcal{C}\cap\mathcal{F}\hookrightarrow\mathcal{A}\rightarrow\mathrm{Ho}(\mathcal{M})$ induces a triangle equivalence $\mathrm{Ho}(\mathcal{M})\simeq(\mathcal{C}\cap\mathcal{F})/(\mathcal{C}\cap\mathcal{F}\cap\mathcal{W})$. Here, $(\mathcal{C}\cap\mathcal{F})/(\mathcal{C}\cap\mathcal{F}\cap\mathcal{W})$ is the stable category of $\mathcal{C}\cap\mathcal{F}$ modulo $\mathcal{C}\cap\mathcal{F}\cap\mathcal{W}.$
	\end{enumerate}
\end{thm}
\subsection{The weakly projective model structures}\label{subsec2.7}

Different from that a Hovey triple involves two complete cotorsion pairs, A. Beligiannis and I. Reiten give a construction of model structures on abelian categories, from only one hereditary complete cotorsion pair; see \cite[Chapter~VIII]{BR}. It is extended to weakly idempotent complete exact categories in \cite{CLZ}.
	
Let $\mathcal{A}$ be an abelian category, $\mathcal{X}$ and $\mathcal{Y}$ two classes of objects of $\mathcal{A}$ which are closed under direct summands and isomorphims. Let $\omega:=\mathcal{X}\cap\mathcal{Y}$. Define the following morphisms classes.
\[\begin{aligned}
	\mathrm{CoFib}_\omega&:=\{\text{monomorphisms with cokernels in }\mathcal{X}\},\\
	\mathrm{Fib}_\omega&:=\{f\;|\;\mathrm{Hom}_\mathcal{A}(W,f) \text{ is surjective for any object }W\in\omega\},\\
	\mathrm{Weq}_{\omega}&:=\left\{p\circ i\;\left|\;\begin{minipage}{0.5\textwidth}
		$i$ is split monic with $\mathrm{Coker}(i)\in\omega$, and
		$p$ is epic with $\mathrm{Ker}(p)\in\mathcal{Y}$.
	\end{minipage}\right\}\right.
\end{aligned}\]
\begin{thm}\label{corr2}$($\cite[Chapter~VIII~Theorems~~4.2 and~4.6]{BR}; $\rm{also}$ \cite[Theorems~1.1 and~1.3]{CLZ}$)$
	Let $\mathcal{A}$ be an abelian category, $S_C$ the class of contravariantly finite hereditary complete cotorsion pairs in $\mathcal{A}$, and $S_M$ the class of weakly projective model structures on $\mathcal{A}$. Then the map $S_C\rightarrow S_M,\;(\mathcal{X,Y})\mapsto\mathcal{M}_\omega:=(\mathrm{CoFib}_\omega,\mathrm{Fib}_\omega,\mathrm{Weq}_\omega)$ is a bijection, where $\omega=\mathcal{X}\cap\mathcal{Y}$. In this case, the associated Quillen homotopy category $\mathrm{Ho}(\mathcal{M}_\omega)=\mathcal{A}[\mathrm{Weq}_\omega^{-1}]$ is equivalent to the additive quotient $\mathcal{X}/\omega$.
\end{thm}
\section{Cotorsion pairs in extensions of abelian categories}\label{sec3}

In this section, we will see how cotorsion pairs in the category $\mathcal{B}$ induce cotorsion pairs in the extension of $\mathcal{B}$; See Theorems~\ref{thm1} and ~\ref{thm2}. We also study the completeness in Theorem~\ref{thm5}. Finally, we state the dual versions of the main results for $\zeta$-coextensions of abelian categories without proofs.

\subsection{Special classes over extensions of abelian categories}\label{3.1}
Let $\mathcal{B}\ltimes_\eta\mathsf{F}$ be an $\eta$-extension of $\mathcal{B}$. For $(X,f)\in\mathcal{B}\ltimes_\eta\mathsf{F}$, we denote by $\mathrm{Com}(X,f)$ the sequence $\mathsf{F}^2(X)\xrightarrow{\mathsf{F}(f)-\eta_X}\mathsf{F}(X)\xrightarrow{f}X$.

 For a class $\mathcal{X}$ of objects of $\mathcal{B}$, define
\[\mathsf{U}^{-1}(\mathcal{X}):=\left\{(X,f)\in\mathcal{B}\ltimes_\eta\mathsf{F}\;|\;X\in\mathcal{X}\right\};\]
\[\Delta(\mathcal{X}):=\left\{(X,f)\in\mathcal{B}\ltimes_\eta\mathsf{F}\;|\;\mathrm{Com}(X,f)\text{ is exact and }\mathrm{Coker}(f)\in\mathcal{X}\right\};\]
It is easy to check that $\mathsf{T}(\mathcal{X})\subseteq\Delta(\mathcal{X})$. It is clear that if $\mathsf{F}^2=0$, then 
\[\Delta(\mathcal{X})=\{(X,f)\in\mathcal{B}\ltimes\mathsf{F}\;|\;f\mbox{ is a monomorphism and }\mathrm{Coker}(f)\in\mathcal{X}\}.\]

\subsection{Some lemmas for extension groups} 

The proof of the following lemma is inspired by \cite[Corollary~4.2]{B2}; see also \cite{K}.

\begin{lem}\label{ext1}
	Let $\mathcal{B}\ltimes_\eta\mathsf{F}$ be an $\eta$-extension of $\mathcal{B}$ and $X\in\mathcal{B}$. The following hold:
	\begin{enumerate}[(1)]
		\item $\mathbb{L}_i\mathsf{F}(X)\cong\mathsf{U}\mathbb{L}_i\mathsf{T}(X)$ for all $i\ge 1$.
		\item If $\mathbb{L}_i\mathsf{F}(X)=0$ for $1\le i\le n$, then $\mathrm{Ext}_\mathcal{B\ltimes_\eta\mathsf{F}}^i(\mathsf{T}(X),(Y,f))\cong\mathrm{Ext}_\mathcal{B}^i(X,Y)$ for any $1\le i\le n$ and $(Y,f)\in\mathcal{B}\ltimes_\eta\mathsf{F}$.
	\end{enumerate}
\end{lem}
\begin{proof}
	Let \begin{equation}\label{lemseq}0\rightarrow K\xrightarrow{i}P\xrightarrow{a}X\rightarrow0\end{equation} be an exact sequence in $\mathcal{B}$ with $P$ projective. Applying the functor $\mathsf{F}$ to (\ref{lemseq}) gives the following exact sequence
	\[0\rightarrow\mathbb{L}_1\mathsf{F}(X)\rightarrow\mathsf{F}(K)\xrightarrow{\mathsf{F}(i)}\mathsf{F}(P)\xrightarrow{\mathsf{F}(a)}\mathsf{F}(X)\rightarrow0.\]
	Applying the functor $\mathsf{T}$ to (\ref{lemseq}) gives the following exact sequence
	\[0\rightarrow\mathbb{L}_1\mathsf{T}(X)\rightarrow\mathsf{T}(K)\xrightarrow{\mathsf{T}(i)}\mathsf{T}(P)\xrightarrow{\mathsf{T}(a)}\mathsf{T}(X)\rightarrow0.\]
	Then we have $\mathbb{L}_1\mathsf{F}(X)\cong\Ker(\mathsf{F}(i))$ and $\mathbb{L}_1\mathsf{T}(X)\cong\Ker(\mathsf{T}(i)).$ Since $\mathsf{U}$ is exact, it follows that $\mathsf{U}(\Ker \mathsf{T}(i))\cong\Ker(\mathsf{UT}(i))$. We observe that $\mathsf{UT}(i)=\begin{pmatrix}
		i&0\\
		0	&\mathsf{F}(i)
	\end{pmatrix}.$ Then we have
	\[\mathsf{U}\mathbb{L}_1\mathsf{T}(X)\cong\mathsf{U}(\Ker \mathsf{T}(i))\cong\Ker(\mathsf{UT}(i))\cong\Ker(\mathsf{F}(i))\cong\mathbb{L}_1\mathsf{F}(X).\]
	Let 
	\begin{equation}\label{res}P_\bullet:\;\;\;\cdots\rightarrow P_2\xrightarrow{d_2} P_1\xrightarrow{d_1}P_0\xrightarrow{d_0}X\rightarrow0\end{equation} 
	be a projective resolution of $X$. Define $K_i=\Ker(d_i)$ for all $i\ge 0$. Then for $i\ge 2$, we have
	\[\mathbb{L}_i\mathsf{F}(X)\cong\mathbb{L}_1\mathsf{F}(K_{i-2})\text{ and }\mathbb{L}_i\mathsf{T}(X)\cong\mathbb{L}_1\mathsf{T}(K_{i-2}).\]
	Therefore, we have
	\[\mathsf{U}\mathbb{L}_i\mathsf{T}(X)\cong\mathsf{U}\mathbb{L}_1\mathsf{T}(K_{i-2})\cong\mathbb{L}_1\mathsf{F}(K_{i-2})\cong\mathbb{L}_i\mathsf{F}(X).\]
	Then we are done.
	
	(2) If $\mathbb{L}_i\mathsf{F}(X)=0$ for all $1\le i\le n$, then $\mathbb{L}_i\mathsf{T}(X)=0$ for all $1\le i\le n$ by (1). Applying the functor $\mathsf{T}$ to (\ref{res}) gives an exact sequence
	\begin{equation*}\mathsf{T}(P_\bullet):\;\;\mathsf{T}(P_{n+1})\rightarrow\mathsf{T}(P_n)\rightarrow\cdots\rightarrow\mathsf{T}(P_0)\rightarrow\mathsf{T}(X)\rightarrow0.\end{equation*}
	By Proposition~\ref{pro1}~(3), each $\mathsf{T}(P_i)$ is projective. Using the adjunction $(\mathsf{T,U})$, for $1\le i\le n$, we get
	\begin{equation*}\begin{aligned}\mathrm{Ext}_{\mathcal{B}\ltimes_\eta\mathsf{F}}^i(\mathsf{T}(X),(Y,f))&=H^i(\mathrm{Hom}_{\mathcal{B}\ltimes_\eta\mathsf{F}}(\mathsf{T}(P_\bullet),(Y,f)))\\&\cong H^i(\mathrm{Hom}_{\mathcal{B}}(P_\bullet,\mathsf{U}(Y,f)))\\
			&=\mathrm{Ext}_{\mathcal{B}}^i(X,Y).\end{aligned}\end{equation*}
		Then we are done.
\end{proof}

\begin{lem}\label{Com}
	Let $(X,f)\in\mathcal{B}\ltimes_\eta\mathsf{F}$, $\rho:X\rightarrow\mathrm{Coker}(f)$ the canonical epimorphism and $\pi:\mathsf{F}(X)\rightarrow\mathrm{Coker}(\mathsf{F}(f)-\eta_X)$ the canonical epimorphism. 
	\begin{enumerate}[(1)]
		\item There exists a morphism $\gamma:\mathrm{Coker}(\mathsf{F}(f)-\eta_X)\rightarrow X$ such that $f=\gamma\circ\pi$ and the sequence $\mathrm{Coker}(\mathsf{F}(f)-\eta_X)\xrightarrow{\gamma}X\xrightarrow{\rho}\mathrm{Coker}(f)\rightarrow0$ is exact.
		\item The morphism $\gamma$ is a monomorphism if and only if the sequence $\mathrm{Com}(X,f)$ is exact.
	\end{enumerate}
\end{lem}
\begin{proof}
	(1) Since $(X,f)\in\mathcal{B}\ltimes_\eta\mathsf{F}$, we get $f\circ(\mathsf{F}(f)-\eta_X)=0$. Then there exists a unique morphism $\gamma:\mathrm{Coker}(\mathsf{F}(f)-\eta_X)\rightarrow X$ such that $f=\gamma\circ\pi$.
	
	For any morphism $g:X\rightarrow Y$, if $g\circ\gamma=0$, then $g\circ f=g\circ\gamma\circ\pi=0$. Thus there exist a morphism $s:\mathrm{Coker}(f)\rightarrow Y$ such that $g=s\circ\rho$. Since $\rho$ is an epimorphism, it follows that $\mathrm{Coker}(\gamma)\cong\mathrm{Coker}(f)$. Then the sequence $\mathrm{Coker}(\mathsf{F}(f)-\eta_X)\xrightarrow{\gamma}X\xrightarrow{\rho}\mathrm{Coker}(f)\rightarrow0$ is exact.
	
	(2)It is straightforward.
\end{proof}

\begin{lem}\label{Ext1}
	Let $(X,f)\in\mathcal{B}\ltimes_\eta\mathsf{F}$. If the sequence $\mathrm{Com}(X,f)$ is exact, then $\mathrm{Ext}_{\mathcal{B}\ltimes_\eta\mathsf{F}}^1((X,f),\mathsf{Z}(Y))\cong\mathrm{Ext}_\mathcal{B}^1(\mathrm{Coker}(f),Y)$ for any $Y\in\mathcal{B}$.
\end{lem}
\begin{proof}
	
	By Proposition~\ref{pro1}~(3) and (4), there is an exact sequence
	\begin{equation}\label{seq3}0\rightarrow (K,k)\xrightarrow{\beta}\mathsf{T}(P)\xrightarrow{\alpha}(X,f)\rightarrow0\end{equation}
	with $P$ a projective object in $\mathcal{B}$.
	
	Since $\mathsf{F}$ is right exact, we have the following commutative diagram with exact rows
	\[\begin{tikzcd}[row sep=small, column sep=small]
		\mathsf{F}^2(K) \arrow[d, "\mathsf{F}(k)-\eta_K"] \arrow[r]       & \mathsf{F}^2(P)\oplus\mathsf{F}^3(P) \arrow[d, "\mathsf{F}(t_P)-\eta_{P\oplus\mathsf{F}(P)}"] \arrow[r] & \mathsf{F}^2(X) \arrow[d, "\mathsf{F}(f)-\eta_X"] \arrow[r] & 0 \\
		\mathsf{F}(K) \arrow[r, "\mathsf{F}(\beta)"] \arrow[d, two heads] & \mathsf{F}(P)\oplus\mathsf{F}^2(P) \arrow[r, "\mathsf{F}(\alpha)"] \arrow[d, two heads]                 & \mathsf{F}(X) \arrow[r] \arrow[d, two heads]                & 0 \\
		\mathrm{Coker}(\mathsf{F}(k)-\eta_K) \arrow[r]                    & \mathrm{Coker}(\mathsf{F}(t_P)-\eta_{P\oplus\mathsf{F}(P)}) \arrow[r]                                   & \mathrm{Coker}(\mathsf{F}(f)-\eta_X) \arrow[r]              & 0
	\end{tikzcd}\]
	
	By Lemma~\ref{Com}~(1), we get the following commutative diagram 
	\[\begin{tikzcd}[column sep=small]
		& \mathrm{Coker}(\mathsf{F}(k)-\eta_K) \arrow[d, "\gamma_K"] \arrow[r] & \mathrm{Coker}(\mathsf{F}(t_P)-\eta_{P\oplus\mathsf{F}(P)}) \arrow[d, "\gamma"] \arrow[r] & \mathrm{Coker}(\mathsf{F}(f)-\eta_X) \arrow[d, "\gamma_X"] \arrow[r] & 0 \\
		0 \arrow[r] & K \arrow[r, "\beta"]                                                 & P\oplus\mathsf{F}(P) \arrow[r, "\alpha"]                                                  & X \arrow[r]                                                          & 0
	\end{tikzcd}\]
	Since $\mathrm{Com}(X,f)$ is exact, it follows from Lemma~\ref{Com}~(2) that $\gamma_X$ is a monomorphism. By Lemma~\ref{Com}~(1) and Snake Lemma, we get a short exact sequence
	\begin{equation}\label{seq4}0\rightarrow\mathrm{Coker}(k)\rightarrow P\rightarrow\mathrm{Coker}(f)\rightarrow0.\end{equation}
	Applying $\mathrm{Hom}_\mathcal{B}(-,Y)$ to (\ref{seq4}) and applying $\mathrm{Hom}_{\mathcal{B}\ltimes_\eta\mathsf{F}}(-,\mathsf{Z}(Y))$ to (\ref{seq3}), since $(\mathsf{C,Z})$ is an adjoint pair, we get a commutative diagram with exact rows
	\[\begin{tikzcd}[column sep=small]
		{\mathrm{Hom}_\mathcal{B}(P,Y)} \arrow[d, "\cong"] \arrow[r]                              & {\mathrm{Hom}_\mathcal{B}(\mathrm{Coker}(k),Y)} \arrow[d, "\cong"] \arrow[r]     & {\mathrm{Ext}_\mathcal{B}^1(\mathrm{Coker}(f),Y)} \arrow[d, dashed]  \arrow[r] & 0 \\
		{\mathrm{Hom}_{\mathcal{B}\ltimes_\eta\mathsf{F}}(\mathsf{T}(P),\mathsf{Z}(Y))} \arrow[r] & {\mathrm{Hom}_{\mathcal{B}\ltimes_\eta\mathsf{F}}((K,k),\mathsf{Z}(Y))}\arrow[r] & {\mathrm{Ext}_{\mathcal{B}\ltimes_\eta\mathsf{F}}^1(X,\mathsf{Z}(Y))} \arrow[r]& 0
	\end{tikzcd}\]
	Therefore, we have $\mathrm{Ext}_{\mathcal{B}\ltimes_\eta\mathsf{F}}^1((X,f),\mathsf{Z}(Y))\cong\mathrm{Ext}_\mathcal{B}^1(\mathrm{Coker}(f),Y)$.
\end{proof}

\begin{rem}
	For $\zeta$-coextensions, analogous dual results of the above lemmas also hold, and we omit the details here.  
\end{rem}

\subsection{Constructions of cotorsion pairs}

\begin{thm}\label{thm1}
	Let $\mathcal{B}\ltimes_\eta\mathsf{F}$ be an $\eta$-extension of $\mathcal{B}$. Assume that $(\mathcal{X,Y})$ is a cotorsion pair in $\mathcal{B}$. If $\mathbb{L}_1\mathsf{F}(\mathcal{X})=0$, then $(\prescript{\perp}{}{\mathsf{U}^{-1}(\mathcal{Y})},\mathsf{U}^{-1}(\mathcal{Y}))$ is a cotorsion pair in $\mathcal{B}\ltimes_\eta\mathsf{F}$; and moreover, it is hereditary if and only if so is $(\mathcal{X,Y})$. 
\end{thm}

To prove Theorem~\ref{thm1}, we need the following lemma.
\begin{lem}\label{class}
	 For a class $\mathcal{X}$ of objects of $\mathcal{B}$, if $\mathbb{L}_1\mathsf{F}(\mathcal{X})=0$, then $\mathsf{T}(\mathcal{X})^\perp=\mathsf{U}^{-1}(\mathcal{X}^\perp)$.
\end{lem}
\begin{proof}
	Let $(Y,f)\in \mathsf{T}(\mathcal{X})^\perp$. For any $X\in\mathcal{X}^\perp$, by Lemma~\ref{ext1}, we have
	\[\mathrm{Ext}_\mathcal{B}^1(X,Y)\cong\mathrm{Ext}_{\mathcal{B}\ltimes_\eta\mathsf{F}}^1(\mathsf{T}(X),(Y,f))=0.\]
	Hence $Y\in\mathcal{X}^\perp$. Thus $(Y,f)\in \mathsf{U}^{-1}(\mathcal{X}^\perp)$ and so $\mathsf{T}(\mathcal{X})^\perp\subseteq\mathsf{U}^{-1}(\mathcal{X}^\perp)$. 
	
	Let $(Y,f)\in\mathsf{U}^{-1}(\mathcal{X}^\perp)$. For any $X\in\mathcal{X}$, by Lemma~\ref{ext1}, we have
	\[\mathrm{Ext}_{\mathcal{B}\ltimes_\eta\mathsf{F}}^1(\mathsf{T}(X),(Y,f))\cong\mathrm{Ext}_\mathcal{B}^1(X,Y)=0.\]
	Hence $(Y,f)\in \mathsf{T}(\mathcal{X})^\perp$. Thus $\mathsf{U}^{-1}(\mathcal{X}^\perp)\subseteq\mathsf{T}(\mathcal{X})^\perp$. Therefore, we get $\mathsf{T}(\mathcal{X})^\perp=\mathsf{U}^{-1}(\mathcal{X}^\perp)$.
\end{proof}
	
\begin{proof}[Proof of Theorem~\ref{thm1}]
	It suffices to show $\mathsf{U}^{-1}(\mathcal{Y})=\big(\prescript{\perp}{}{\mathsf{U}^{-1}(\mathcal{Y})}\big)^\perp$. Since $(\mathcal{X,Y})$ is a cotorsion pair, it follows that $\mathcal{Y}=\mathcal{X}^\perp$. By the assumption that $\mathbb{L}_1\mathsf{F}(\mathcal{X})=0$, it follows from Lemma~\ref{class} that 
	\[\mathsf{U}^{-1}(\mathcal{Y})=\mathsf{U}^{-1}(\mathcal{X}^\perp)=\mathsf{T}(\mathcal{X})^\perp.\]
	Thus
	\[\big(\prescript{\perp}{}{\mathsf{U}^{-1}(\mathcal{Y})}\big)^\perp=\big[\prescript{\perp}{}{\big(\mathsf{T}(\mathcal{X})^\perp\big)}\big]^\perp=\mathsf{T}(\mathcal{X})^\perp=\mathsf{U}^{-1}(\mathcal{Y}).\]
	Here, one uses the fact $(\prescript{\perp}{}{(\mathcal{S}^\perp)})^\perp=\mathcal{S}^\perp$ for any class $\mathcal{S}$ of objects.
	
	If $(\mathcal{X},\mathcal{Y})$ is hereditary, then $\mathcal{Y}$ is closed under the cokernels of monomorphisms. Let $(X_1,a_1), (X_2,a_2)\in\mathsf{U}^{-1}(\mathcal{Y})$ and $f:(X_1,a_1)\rightarrow X_2$ be a monomorphism in $\mathcal{B}\ltimes_\eta\mathsf{F}$. Since $\mathsf{U}$ is exact, we get $\mathsf{U}(\mathrm{Coker}(f))\cong\mathrm{Coker}(\mathsf{U}(f))\in\mathcal{Y}$. Thus $\mathrm{Coker}(f)\in\mathsf{U}^{-1}(\mathcal{Y})$. Therefore, $\mathsf{U}^{-1}(f)$ is closed under the cokernels of monomorphisms. Hence, $(\prescript{\perp}{}{\mathsf{U}^{-1}(\mathcal{Y})},\mathsf{U}^{-1}(\mathcal{Y}))$ is hereditary.
	
	Conversely, assume that $(\prescript{\perp}{}{\mathsf{U}^{-1}(\mathcal{Y})},\mathsf{U}^{-1}(\mathcal{Y}))$ is hereditary. Let $Y_1,Y_2\in\mathcal{Y}$ and $g:Y_1\rightarrow Y_2$ be a monomorphism in $\mathcal{B}$. Then $\mathsf{Z}(Y_1),\mathsf{Z}(Y_2)\in\mathsf{U}^{-1}(\mathcal{Y})$ and $\mathsf{Z}(g)$ is a monomorphism in $\mathcal{B}\ltimes_\eta\mathsf{F}$. Thus $\mathrm{Coker}(\mathsf{Z}(g))\in\mathsf{U}^{-1}(\mathcal{Y})$. Then we have
	\[\mathrm{Coker}(g)\cong\mathrm{Coker}(\mathsf{UZ}(g))\cong\mathsf{U}(\mathrm{Coker}(\mathsf{Z}(g)))\in\mathcal{Y}.\]
	Therefore, $(\mathcal{X,Y})$ is hereditary.
\end{proof}

\begin{thm}\label{thm2}
	Let $\mathcal{B}\ltimes \mathsf{F}$ be a right trivial extension of $\mathcal{B}$ with $\mathsf{F}^2=0$. Assume that $(\mathcal{X,Y})$ is a cotorsion pair in $\mathcal{B}$. Then $(\Delta(\mathcal{X}),\Delta(\mathcal{X})^\perp)$ is a cotorsion pair in $\mathcal{B}\ltimes\mathsf{F}$. Moreover, if $\mathbb{L}_1\mathsf{F}(\mathcal{X})=0$, then $(\Delta(\mathcal{X}),\Delta(\mathcal{X})^\perp)$ is hereditary if and only if so is $(\mathcal{X,Y})$.
\end{thm}

To prove Theorem~\ref{thm2}, we need some preparations.

\begin{lem}\label{class2}
	Let $\mathcal{B}\ltimes \mathsf{F}$ be a right trivial extension of $\mathcal{B}$ with $\mathsf{F}^2=0$. For a class $\mathcal{X}$ of objects of $\mathcal{B}$, if $\mathcal{X}$ contains all injective objects of $\mathcal{B}$, then $\Delta(\prescript{\perp}{}{\mathcal{X}})=\prescript{\perp}{}{\mathsf{Z}(\mathcal{X})}$.
\end{lem}
\begin{proof}
	Let $(Y,f)\in\Delta(\prescript{\perp}{}{\mathcal{X}})$. By definition, $f$ is a monomorphism and $\mathrm{Coker}(f)\in\prescript{\perp}{}{\mathcal{X}}$. Since $f$ is a monomorphism and $\mathrm{Ext}_\mathcal{B}^1(\mathsf{C}(Y,f),\mathcal{X})=\mathrm{Ext}_\mathcal{B}^1(\mathrm{Coker}(f),\mathcal{X})=0$, it follows that $\mathrm{Ext}^1_{\mathcal{B}\ltimes\mathsf{F}}((Y,f),\mathsf{Z}(\mathcal{X}))=0$, that is, $(Y,f)\in \prescript{\perp}{}{\mathsf{Z}(\mathcal{X})}.$ Therefore, we have  $\Delta(\prescript{\perp}{}{\mathcal{X}})\subseteq\prescript{\perp}{}{\mathsf{Z}(\mathcal{X})}$.
	
	On the other hand, let $(Y,f)\in\prescript{\perp}{}{\mathsf{Z}(\mathcal{X})}$. Then $\mathrm{Ext}^1_{\mathcal{B}\ltimes\mathsf{F}}((Y,f),\mathsf{Z}(\mathcal{X}))=0.$
	
	\textbf{Step 1.} We show that for any $X\in\mathcal{X}$, $\mathrm{Hom}_\mathcal{B}(f,X):\mathrm{Hom}_\mathcal{B}(Y,X)\rightarrow\mathrm{Hom}_\mathcal{B}(F(Y),X)$ is an epimorphism. For any $g\in\mathrm{Hom}_\mathcal{B}(\mathsf{F}(Y),X)$, we have the following  commutative diagram
	\[\begin{tikzcd}
		& \mathsf{F}(X) \arrow[d, "0"] \arrow[r, "{\left(\begin{smallmatrix}1\\0\end{smallmatrix}\right)}"] & \mathsf{F}(X)\oplus\mathsf{F}(Y) \arrow[d, "{\left(\begin{smallmatrix}0\;\;g\\0\;\;f\end{smallmatrix}\right)}"] \arrow[r, "{(0,1)}"] & \mathsf{F}(Y) \arrow[d, "f"] &   \\
		0 \arrow[r] & X \arrow[r, "{\left(\begin{smallmatrix}1\\0\end{smallmatrix}\right)}"]                           & X\oplus Y \arrow[r, "{(0,1)}"]                                                                          & Y \arrow[r]                  & 0
	\end{tikzcd}\]
	By the assumption $\mathsf{F}^2=0$, $(V,v)\in\mathcal{B}\ltimes \mathsf{F}$ for any $V\in\mathcal{B}$ and $v\in\mathrm{Hom}_\mathcal{B}(\mathsf{F}(V),V)$. Therefore, we have the following short exact sequence in $\mathcal{B}\ltimes\mathsf{F}$:
	\begin{equation}\label{seq2}0\rightarrow\mathsf{Z}(X)\xrightarrow{\left(\begin{smallmatrix}
			1\\ 0
		\end{smallmatrix}\right)}\big(X\oplus Y,\left(\begin{smallmatrix}
		0&g\\ 0&f
	\end{smallmatrix}\right)\big)\xrightarrow{(0,1)}(Y,f)\rightarrow0.\end{equation}
Since $\mathrm{Ext}^1_{\mathcal{B}\ltimes\mathsf{F}}((Y,f),\mathsf{Z}(X))=0,$ the sequence (\ref{seq2}) splits. Therefore, there is a morphism $\big(\begin{smallmatrix}
	\alpha\\ \beta
\end{smallmatrix}\big):(Y,f)\rightarrow\big(X\oplus Y,\left(\begin{smallmatrix}
0&g\\ 0&f
\end{smallmatrix}\right)\big)$ in $\mathcal{B}\ltimes\mathsf{F}$ such that $(0,1)\big(\begin{smallmatrix}
\alpha\\ \beta
\end{smallmatrix}\big)=\mathsf{Id}_{(Y,f)}$. Thus $\beta=\mathsf{Id}_Y$. There is a commutative diagram
\[\begin{tikzcd}
	\mathsf{F}(Y) \arrow[d, "f"] \arrow[r, "{\big(\begin{smallmatrix}\mathsf{F}(\alpha)\\1\end{smallmatrix}\big)}"] & \mathsf{F}(X)\oplus\mathsf{F}(Y) \arrow[d, "\left(\begin{smallmatrix}0\;\;g\\0\;\;f\end{smallmatrix}\right)"] \\
	Y \arrow[r, "{\big(\begin{smallmatrix}\alpha\\1\end{smallmatrix}\big)}"]                                & X\oplus Y                                                                                                    
	\end{tikzcd}\]
	Therefore, we get $g=\alpha\circ f$. Then $\mathrm{Hom}_\mathcal{B}(f,X)$ is an epimorphism.
	
	\textbf{Step 2.} We show that $f$ is a monomorphism. ​Since the abelian category $\mathcal{B}$ has enough injective objects, there exists a monomorphism $i:\mathsf{F}(Y)\rightarrow I$ with $I$ injective.​ By the assumption $I\in\mathcal{X}$,  $\mathrm{Hom}_\mathcal{B}(f,I):\mathrm{Hom}_\mathcal{B}(Y,I)\rightarrow\mathrm{Hom}_\mathcal{B}(F(Y),I)$ is an epimorphism. Thus there is $v\in\mathrm{Hom}_\mathcal{B}(Y,I)$ such that $v\circ f=i$. Therefore, $f$ is a monomorphism.
	
	\textbf{Step 3.} We show that $\mathrm{Coker}(f)\in\prescript{\perp}{}{\mathcal{X}}$. Since $f$ is a monomorphism, by Lemma~\ref{Ext1}, we have
	\[\mathrm{Ext}^1_\mathcal{B}(\mathrm{Coker}(f),\mathcal{X})=\mathrm{Ext}^1_\mathcal{B}(\mathsf{C}(Y,f),\mathcal{X})\cong\mathrm{Ext}^1_{\mathcal{B}\ltimes\mathsf{F}}((Y,f),\mathsf{Z}(\mathcal{X}))=0.\]
	Therefore, we get $\mathrm{Coker}(f)\in\prescript{\perp}{}{\mathcal{X}}$. So $(Y,f)\in\Delta(\prescript{\perp}{}{\mathcal{X}})$. Hence $\prescript{\perp}{}{\mathsf{Z}(\mathcal{X})}\subseteq\Delta(\prescript{\perp}{}{\mathcal{X}}).$ Then we are done.
\end{proof}

\begin{lem}\label{closed}
	Let $\mathcal{B}\ltimes \mathsf{F}$ be right trivial extension of $\mathcal{B}$, $\mathcal{X}$ a class of objects of $\mathcal{B}$. If  $\mathbb{L}_1\mathsf{F}(\mathcal{X})=0$ , then $\Delta(\mathcal{X})$ is closed under the kernels of epimorphisms if and only if $\mathcal{X}$ is closed under the kernels of epimorphisms.
\end{lem}
\begin{proof}
	Assume that $\Delta(\mathcal{X})$ is closed under the kernels of epimorphisms. Let \begin{equation}\label{seq}0\rightarrow X_1\rightarrow X_2\rightarrow X_3\rightarrow0\end{equation} be an exact sequence with $X_2,X_3\in\mathcal{X}$. Since $\mathbb{L}_1\mathsf{F}(\mathcal{X})=0$, it follows from Lemma~\ref{ext1}~(1) that $\mathbb{L}_1\mathsf{T}(\mathcal{X})=0.$ Then applying the functor $\mathsf{T}$ to (\ref{seq}) gives an exact sequence
	\[0\rightarrow\mathsf{T}(X_1)\rightarrow\mathsf{T}(X_2)\rightarrow\mathsf{T}(X_3)\rightarrow0.\]
	Since  $\mathsf{T}(\mathcal{X})\subseteq\Delta(\mathcal{X})$, one gets $\mathsf{T}(X_2),\mathsf{T}(X_3)\in\Delta(\mathcal{X})$. Thus $\mathsf{T}(X_1)\in\Delta(\mathcal{X})$ and so $X_1=\mathsf{C}(\mathsf{T}(X_1))\in\mathcal{X}$. This prove that $\mathcal{X}$ is closed under the kernels of epimorphisms.
	
	Conversely, assume that $\mathcal{X}$ is closed under the kernels of epimorphisms. Let 
	$$0\rightarrow(X_1,f_1)\xrightarrow{\alpha}(X_2,f_2)\xrightarrow{\beta}(X_3,f_3)\rightarrow0$$ be an exact sequence with $(X_2,f_2),(X_3,f_3)\in\Delta(\mathcal{X})$. Thus $\mathrm{Coker}(f_2),\mathrm{Coker}(f_3)\in\mathcal{X}$ and $\mathrm{Com}(X_2,f_2),\mathrm{Com}(X_3,f_3)$ are exact. Since $\mathsf{F}$ is right exact, it follows from Lemma~\ref{Com} that there is a commutative diagram
	\[\begin{tikzcd}[column sep=small]
		& \mathsf{F}(\mathrm{Coker}(f_1)) \arrow[r,"\widetilde{\alpha}"] \arrow[d, "\gamma_1"] & \mathsf{F}(\mathrm{Coker}(f_2)) \arrow[r] \arrow[d, "\gamma_2"] & \mathsf{F}(\mathrm{Coker}(f_3)) \arrow[r] \arrow[d, "\gamma_3"] & 0 \\
		0 \arrow[r] & X_1 \arrow[r,"\alpha"]                                                   & X_2 \arrow[r,"\beta"]                                                   & X_3 \arrow[r]                                                   & 0 
	\end{tikzcd}\]
	with $\gamma_2,\gamma_3$ monomorphisms. By Snake Lemma and Lemma~\ref{Com}, we get an exact sequence
	\[0\rightarrow\mathrm{Coker}(f_1)\rightarrow\mathrm{Coker}(f_2)\rightarrow\mathrm{Coker}(f_3)\rightarrow0.\]
	Since $\mathcal{X}$ is closed under the  kernels of epimorphisms, $\mathrm{Coker}(f_1)\in\mathcal{X}$. Since $\mathrm{Coker}(f_3)\in\mathcal{X}$ and $\mathbb{L}_1\mathsf{F}(\mathcal{X})=0$,  $\widetilde{\alpha}$ is a monomorphism. Then $\gamma_1$ is a monomorphism. By Lemma~\ref{Com}~(2), one gets that $\mathrm{Com}(X_1,f_1)$ is exact. Therefore, $(X_1,f_1)\in\Delta(\mathcal{X})$. This prove that $\Delta(\mathcal{X})$ is closed under the kernels of epimorphisms. 
\end{proof}

\begin{proof}[Proof of Theorem~\ref{thm2}]
	It suffices to show that $\Delta(\mathcal{X})=\prescript{\perp}{}{(\Delta(\mathcal{X})^\perp)}$. Since $(\mathcal{X,Y})$ is a cotorsion pair in $\mathcal{B}$, $\mathcal{Y}$ contains all injective objects of $\mathcal{B}$ and $\mathcal{X}=\prescript{\perp}{}{\mathcal{Y}}$. By Lemma~\ref{class2}, we have
	\[\Delta(\mathcal{X})=\Delta(\prescript{\perp}{}{\mathcal{Y}})=\prescript{\perp}{}{\mathsf{Z}(\mathcal{X})}.\]
	Thus
	\[\prescript{\perp}{}{(\Delta(\mathcal{X})^\perp)}=\prescript{\perp}{}{\big[\big(\prescript{\perp}{}{\mathsf{Z}(\mathcal{X})}\big)^\perp\big]}=\prescript{\perp}{}{\mathsf{Z}(\mathcal{X})}=\Delta(\mathcal{X}).\]
	
	If $\mathbb{L}_1\mathsf{F}(\mathcal{X})=0$, then by Lemma~\ref{closed}, $\Delta(\mathcal{X})$ is closed under the kernels of epimorphisms if and only if $\mathcal{X}$ is closed under the kernels of epimorphisms. That is, $(\Delta(\mathcal{X}),\Delta(\mathcal{X})^\perp)$ is hereditary if and only if so is $(\mathcal{X,Y})$.
\end{proof}

By comparing cotorsion pairs given in Theorem~\ref{thm1} and Theorem~\ref{thm2}, we have the following relations.

\begin{pro}\label{pro3.10}
	Let $\mathcal{B}\ltimes \mathsf{F}$ be right trivial extension of $\mathcal{B}$ with $\mathsf{F}^2=0$. Assume that $(\mathcal{X,Y})$ is a cotorsion pair in $\mathcal{B}$. If $\mathbb{L}_1\mathsf{F}(\mathcal{X})=0$, then the cotorsion pairs $(\prescript{\perp}{}{\mathsf{U}^{-1}(\mathcal{Y})},\mathsf{U}^{-1}(\mathcal{Y}))$ and $(\Delta(\mathcal{X}),\Delta(\mathcal{X})^\perp)$ in $\mathcal{B}\ltimes\mathsf{F}$ have a relation $\Delta(\mathcal{X})^\perp\subseteq\mathsf{U}^{-1}(\mathcal{Y})$, or equivalently, $\prescript{\perp}{}{\mathsf{U}^{-1}(\mathcal{Y})}\subseteq\Delta(\mathcal{X})$.
\end{pro}
\begin{proof}
	By Theorem~\ref{thm1} and Theorem~\ref{thm2}, we have cotorsion pairs $(\prescript{\perp}{}{\mathsf{U}^{-1}(\mathcal{Y})},\mathsf{U}^{-1}(\mathcal{Y}))$ and $(\Delta(\mathcal{X}),\Delta(\mathcal{X})^\perp)$. 
	
	We will show that $\Delta(\mathcal{X})^\perp\subseteq\mathsf{U}^{-1}(\mathcal{Y})$. By Lemma~\ref{class2}, we have
	\[\Delta(\mathcal{X})^\perp=\Delta(\prescript{\perp}{}{\mathcal{Y}})^\perp=\big(\prescript{\perp}{}{\mathsf{Z}(\mathcal{Y})}\big)^\perp.\]
	Since $\mathbb{L}_`\mathsf{F}(\mathcal{X})=0$, it follows from Lemma~\ref{class} that 
	\[\mathsf{U}^{-1}(\mathcal{Y})=\mathsf{U}^{-1}(\mathcal{X}^\perp)=\mathsf{T}(\mathcal{X})^\perp.\]
	Thus, to prove $\Delta(\mathcal{X})^\perp\subseteq\mathsf{U}^{-1}(\mathcal{Y})$, it suffices to show $\mathsf{T}(\mathcal{X})\subseteq\prescript{\perp}{}{\mathsf{Z}(\mathcal{Y})}$.
	
	For any $X\in\mathcal{B}$, $\mathrm{Com}(\mathsf{T}(X),t_X)$ is exact, since $\eta=0$. By Lemma~\ref{Ext1}, we have
	\[\mathrm{Ext}_{\mathcal{B}\ltimes\mathsf{F}}^1(\mathsf{T}(\mathcal{X}),\mathsf{Z}(\mathcal{Y}))\cong\mathrm{Ext}_\mathcal{B}^1(\mathsf{C}\mathsf{T}(\mathcal{X}),\mathcal{Y})=\mathrm{Ext}_\mathcal{B}^1(\mathcal{X,Y})=0.\]
	So $\mathsf{T}(\mathcal{X})\subseteq\prescript{\perp}{}{\mathsf{Z}(\mathcal{Y})}$. Then $\Delta(\mathcal{X})^\perp\subseteq\mathsf{U}^{-1}(\mathcal{Y})$.
\end{proof}

\begin{thm}\label{thm4}
	Let $\mathcal{B}\ltimes \mathsf{F}$ be a right trivial extension of $\mathcal{B}$ with $\mathsf{F}^2=0$. Assume that $(\mathcal{X,Y})$ is a cotorsion pair in $\mathcal{B}$ and $\mathbb{L}_1\mathsf{F}(\mathcal{X})=0$. If $\mathrm{Ext}_\mathcal{B}^1(X,\mathsf{F}(X))=0$ for any $X\in\mathcal{X}$, then $$\Delta(\mathcal{X})^\perp=\mathsf{U}^{-1}(\mathcal{Y})\;\; ,\;\; \prescript{\perp}{}{\mathsf{U}^{-1}(\mathcal{Y})}=\Delta(\mathcal{X})=\{\mathsf{T}(X)\;|\;X\in\mathcal{X}\},$$
	and $(\mathsf{T}(\mathcal{X}),\mathsf{U}^{-1}(\mathcal{Y}))$ is a cotorsion pair in $\mathcal{B}\ltimes\mathsf{F}$. 
\end{thm}

To prove Theorem~\ref{thm4}, we need the following characterization of $\Delta(\mathcal{X})$.

\begin{lem}\label{DT}
	Let $\mathcal{B}\ltimes\mathsf{F}$ be a right trivial extension of an abelian category $\mathcal{B}$. For a class $\mathcal{X}$ of objects of $\mathcal{B}$, if $\mathrm{Ext}_\mathcal{B}^1(X,\mathsf{F}(X))=0$ for any $X\in\mathcal{X}$, then $\Delta(\mathcal{X})=\mathsf{T}(\mathcal{X})$.
\end{lem}
\begin{proof}
	 We mention that $\mathsf{T}(\mathcal{X})\subseteq\Delta(\mathcal{X})$. To see that they are equal, it remains to prove that $\Delta(\mathcal{X})\subseteq\mathsf{T}(\mathcal{X})$.
	
	Let $(X,f)\in\Delta(\mathcal{X})$. Then $\mathrm{Com}(X,f)$ is exact and $\mathrm{Coker}(f)\in\mathcal{X}$. Since $\mathsf{F}$ is right exact and $\eta=0$, by Lemma~\ref{Com}, there is an exact sequence
	\begin{equation}\label{seq0}0\longrightarrow\mathsf{F}(\mathrm{Coker}(f))\xlongrightarrow{\gamma} X\xlongrightarrow{\rho}\mathrm{Coker}(f)\rightarrow0.\end{equation}
	Since $\mathrm{Ext}_\mathcal{B}^1(X,\mathsf{F}(X))=0$ for any $X\in\mathcal{X}$, the sequence (\ref{seq0}) splits. Therefore, there are morphisms $u:\mathrm{Coker}(f)\rightarrow X$ and $v:X\rightarrow\mathsf{F}(\mathrm{Coker}(f))$ such that 
	\[\rho\circ u=1,\;v\circ\gamma=1\;\mbox{ and }\;u\circ\rho+\gamma\circ v=1_X.\]
	Then $\big(\begin{smallmatrix}
		\rho\\ v
	\end{smallmatrix}\big):(X,f)\rightarrow \mathsf{T}(\mathrm{Coker}(f))$ is an isomorphism in $\mathcal{B}\ltimes\mathsf{F}$. Thus $(X,f)\in\mathsf{T}(\mathcal{X})$ and so $\Delta(\mathcal{X})\subseteq\mathsf{T}(\mathcal{X})$.
\end{proof}

\begin{proof}[Proof of Theorem~\ref{thm4}]
	 Since $\mathrm{Ext}_\mathcal{B}^1(X,\mathsf{F}(X))=0$ for any $X\in\mathcal{X}$, it follows from Lemma~\ref{DT} that $\Delta(\mathcal{X})=\mathsf{T}(\mathcal{X})$.	
	By Proposition~\ref{pro3.10}, one has $\prescript{\perp}{}{\mathsf{U}^{-1}(\mathcal{Y})}\subseteq\Delta(\mathcal{X})$. On the other hand, let $X\in\mathcal{X}$. For any $(Y,f)\in\mathsf{U}^{-1}(\mathcal{Y})$, since $\mathbb{L}_1\mathsf{F}(\mathcal{X})=0$, it follows from Lemma~\ref{ext1} that
	\[\mathrm{Ext}_{\mathcal{B}\ltimes\mathsf{F}}^1(\mathsf{T}(X),(Y,f))\cong\mathrm{Ext}_\mathcal{B}^1(X,Y)=0.\]
	Thus $\mathsf{T}(X)\in\prescript{\perp}{}{\mathsf{U}^{-1}(\mathcal{Y})}$ and so $\Delta(\mathcal{X})=\mathsf{T}(\mathcal{X})\subseteq \prescript{\perp}{}{\mathsf{U}^{-1}(\mathcal{Y})}$. Then we get $\Delta(\mathcal{X})=\mathsf{T}(\mathcal{X})= \prescript{\perp}{}{\mathsf{U}^{-1}(\mathcal{Y})}$. Therefore, $(\Delta(\mathcal{X}),\Delta(\mathcal{X})^\perp)=(\prescript{\perp}{}{\mathsf{U}^{-1}(\mathcal{Y})},\mathsf{U}^{-1}(\mathcal{Y}))$ and $(\mathsf{T}(\mathcal{X}),\mathsf{U}^{-1}(\mathcal{Y}))$ is a cotorsion pair in $\mathcal{B}\ltimes\mathsf{F}$. 
\end{proof}

\subsection{Completeness of cotorsion pairs}

The significance of the completeness of the cotorsion pair is not only reflected in the theory itself, but also in the abelian model structure.

\begin{thm}\label{thm5}
	Let $\mathcal{B}\ltimes\mathsf{F}$ be a right trivial extension of an abelian category $\mathcal{B}$. Assume that $(\mathcal{X,Y})$ is a hereditary complete cotorsion pair. If $\mathbb{L}_1\mathsf{F}(\mathcal{X})=0$ and $\mathsf{F}(\mathcal{X})\subseteq\mathcal{Y}$, then $(\prescript{\perp}{}{\mathsf{U}^{-1}(\mathcal{Y})},\mathsf{U}^{-1}(\mathcal{Y}))$ is a hereditary complete cotorsion pair.
\end{thm}

 For the proof of Theorem~\ref{thm5}, we need the following easy observations.

\begin{lem}\label{can-seq}
	Let $\mathcal{B}\ltimes_\eta\mathsf{F}$ be an $\eta$-extension of an abelian category $\mathcal{B}$. Let $(X,f)\in\mathcal{B}\ltimes_\eta\mathsf{F}$, $\rho:X\rightarrow\mathrm{Coker}(f)$ be the canonical epimorphism and $i:\mathrm{Im}(f)\rightarrow X$ be the canonical monomorphism. Then there exists a morphism $g:\mathsf{F}(\mathrm{Im}f)\rightarrow\mathrm{Im}(f)$ such that
	\[0\rightarrow(\mathrm{Im}(f),g)\xrightarrow{i}(X,f)\xrightarrow{\rho}(\mathrm{Coker}(f),0)\rightarrow0\]
	is an exact sequence in $\mathcal{B}\ltimes_\eta\mathsf{F}$. 
	Moreover, if $\eta=0$, then $g=0.$
\end{lem}
\begin{proof}
	There is a decomposition $f = k\circ i$ with $k$ an epimorphism. Let $g=k\circ\mathsf{F}(i)$. Then we have $i\circ g=f\circ\mathsf{F}(i)$ and
	\begin{equation}\label{3.4}\begin{aligned}
		i\circ g\circ \mathsf{F}(g)&=i\circ k\circ \mathsf{F}(i)\circ\mathsf{F}(g)
		=f\circ\mathsf{F}(i\circ g)\\
		&=f\circ\mathsf{F}(f)\circ\mathsf{F}^2(i)
		=f\circ\eta_X\circ\mathsf{F}^2(i)\\
		&=f\circ\mathsf{F}(i)\circ\eta_{\mathrm{Im}(f)}
		=i\circ g\circ\eta_{\mathrm{Im}(f)}.
	\end{aligned}\end{equation}
Since $i$ is a monomorphism, we get $g\circ\mathsf{F}(g)=g\circ\eta_{\mathrm{Im}(f)}$. Therefore, $(\mathrm{Im}(f),g)\in\mathcal{B}\ltimes_\eta\mathsf{F}$ and $0\rightarrow(\mathrm{Im}(f),g)\xrightarrow{i}(X,f)\xrightarrow{\rho}(\mathrm{Coker}(f),0)\rightarrow0$ is an exact sequence in $\mathcal{B}\ltimes_\eta\mathsf{F}$. Moreover, if $\eta=0$, then we have
\[
	i\circ g\circ \mathsf{F}(k)=f\circ\mathsf{F}(i)\circ\mathsf{F}(k)=f\circ\mathsf{F}(f)=0.
\]
Since $i$ is a monomorphism and $\mathsf{F}(k)$ is an epimorphism, we get $g=0$. 
\end{proof}

\begin{proof}[Proof of Theorem~\ref{thm5}]
	 By Theorem~\ref{thm1}, $(\prescript{\perp}{}{\mathsf{U}^{-1}(\mathcal{Y})},\mathsf{U}^{-1}(\mathcal{Y}))$ is a hereditary cotorsion pair in $\mathcal{B}\ltimes\mathsf{F}$.
	
	Let $(X,f)\in\mathcal{B}\ltimes\mathsf{F}$. By Lemma~\ref{can-seq}, there is an exact sequence \begin{equation}\label{s}0\longrightarrow\mathsf{Z}(\mathrm{Im}(f))\longrightarrow(X,f)\longrightarrow\mathsf{Z}(\mathrm{Coker}(f))\longrightarrow0.\end{equation}
	Since $(\mathcal{X,Y})$ is a complete cotorsion pair, there is an exact sequence $$0\longrightarrow\mathrm{Coker}(f)\xlongrightarrow{i_1} Y_1\xlongrightarrow{\pi_1} X_1\longrightarrow0$$
	with $Y_1\in\mathcal{Y}$ and $X_1\in\mathcal{X}$. Then the sequence
	$$
	0\longrightarrow\mathrm{Coker}(f)\xlongrightarrow{\big(\begin{smallmatrix}i_1\\ 0\end{smallmatrix}\big)} Y_1\oplus\mathsf{F}(X_1)\xlongrightarrow{\big(\begin{smallmatrix}\pi_1&0\\ 0&1\end{smallmatrix}\big)} X_1\oplus\mathsf{F}(X_1)\longrightarrow0
	$$
	is exact. Let $\beta=\begin{pmatrix}
		0&0\\ \mathsf{F}(\pi_1)&0
	\end{pmatrix}$. Then we have the following commutative diagram
	\[\begin{tikzcd}
		& \mathsf{F}(\mathrm{Coker}(f)) \arrow[d, "0"] \arrow[r] & \mathsf{F}(Y_1)\oplus\mathsf{F}^2(X_1) \arrow[d, "\beta"] \arrow[r]                                     & \mathsf{F}(X_1)\oplus\mathsf{F}^2(X_1) \arrow[d,"t_{X_1}"] &   \\
		0 \arrow[r] & \mathrm{Coker}(f) \arrow[r, "\big(\begin{smallmatrix}i_1\\ 0\end{smallmatrix}\big)"]                                        & Y_1\oplus\mathsf{F}(X_1) \arrow[r, "\big(\begin{smallmatrix}\pi_1\;\;0\\ 0\;\;\;1\end{smallmatrix}\big)"] & X_1\oplus\mathsf{F}(X_1) \arrow[r]                                                                                & 0
	\end{tikzcd}\]
	 Since $\beta\circ\mathsf{F}(\beta)=0$, we get $(Y_1\oplus\mathsf{F}(X_1),\beta)\in\mathcal{B}\ltimes\mathsf{F}$.	Thus, the sequence $$0\rightarrow\mathsf{Z}(\mathrm{Coker}(f))\rightarrow(Y_1\oplus\mathsf{F}(X_1),\beta)\rightarrow\mathsf{T}(X_1)\rightarrow0$$ is exact in $\mathcal{B}\ltimes\mathsf{F}$. By Lemma~\ref{class}, we have $\mathsf{T}(\mathcal{X})\subseteq\prescript{\perp}{}{\mathsf{U}^{-1}(\mathcal{X}^\perp)}=\prescript{\perp}{}{\mathsf{U}^{-1}(\mathcal{Y})}$. Then $\mathsf{T}(X_1)\in\prescript{\perp}{}{\mathsf{U}^{-1}(\mathcal{Y})}$. Since $\mathsf{F}(\mathcal{X})\subseteq\mathcal{Y}$, we get $(Y_1\oplus\mathsf{F}(X_1),\beta)\in\mathsf{U}^{-1}(\mathcal{Y})$. Therefore, $\mathsf{Z}(\mathrm{Coker}(f))$ has a special left $\mathsf{U}^{-1}(\mathcal{Y})$-approximation.
	
	Similarly, there is an exact sequence $$0\longrightarrow\mathrm{Im}(f)\xlongrightarrow{i_2} Y_2\xlongrightarrow{\pi_2} X_2\longrightarrow0$$
	with $Y_2\in\mathcal{Y}$ and $X_2\in\mathcal{X}$. Then the sequence
	$$
	0\longrightarrow\mathrm{Im}(f)\xlongrightarrow{\big(\begin{smallmatrix}i_2\\ 0\end{smallmatrix}\big)} Y_2\oplus\mathsf{F}(X_2)\xlongrightarrow{\big(\begin{smallmatrix}\pi_2&0\\ 0&1\end{smallmatrix}\big)} X_2\oplus\mathsf{F}(X_2)\longrightarrow0
	$$
	is exact. Let $\alpha=\begin{pmatrix}
		0&0\\ \mathsf{F}(\pi_2)&0
	\end{pmatrix}$. Then we have the following commutative diagram
	\[\begin{tikzcd}
		& \mathsf{F}(\mathrm{Im}(f)) \arrow[d, "0"] \arrow[r] & \mathsf{F}(Y_2)\oplus\mathsf{F}^2(X_2) \arrow[d, "\alpha"] \arrow[r]                                     & \mathsf{F}(X_2)\oplus\mathsf{F}^2(X_2) \arrow[d,"t_{X_2}"] &   \\
		0 \arrow[r] & \mathrm{Im}(f) \arrow[r, "\big(\begin{smallmatrix}i_2\\ 0\end{smallmatrix}\big)"]                                        & Y_2\oplus\mathsf{F}(X_2) \arrow[r, "\big(\begin{smallmatrix}\pi_2\;\;0\\ 0\;\;\;1\end{smallmatrix}\big)"] & X_2\oplus\mathsf{F}(X_2) \arrow[r]                                                                                & 0
	\end{tikzcd}\]
	Thus, the sequence $0\rightarrow\mathsf{Z}(\mathrm{Im}(f))\rightarrow(Y_2\oplus\mathsf{F}(X_2),\alpha)\rightarrow\mathsf{T}(X_2)\rightarrow0$ is exact. Then $\mathsf{Z}(\mathrm{Im}(f))$ has a special left $\mathsf{U}^{-1}(\mathcal{Y})$-approximation.
	
	By (\ref{s}) and Proposition~\ref{cover}, $(X,f)$ has a special left $\mathsf{U}^{-1}(\mathcal{Y})$-approximation. This proves the completeness of $(\prescript{\perp}{}{\mathsf{U}^{-1}(\mathcal{Y})},\mathsf{U}^{-1}(\mathcal{Y}))$.
\end{proof}

\begin{cor}\label{cor1}
	Let $\mathcal{B}\ltimes\mathsf{F}$ be a right trivial extension of $\mathcal{B}$ with $\mathsf{F}^2=0$. Assume that $(\mathcal{X,Y})$ is a hereditary complete cotorsion pair. If $\mathbb{L}_1(\mathcal{X})=0$ and $\mathsf{F}(\mathcal{X})\subseteq\mathcal{Y}$, then $(\mathsf{T}(\mathcal{X}),\mathsf{U}^{-1}(\mathcal{Y}))$ is a hereditary complete cotorsion pair.
\end{cor}
\begin{proof}
	Since $\mathsf{F}(\mathcal{X})\subseteq\mathcal{Y}$, we have $\mathrm{Ext}_\mathcal{B}^1(X,\mathsf{F}(X))=0$ for all $X\in\mathcal{X}$. By Theorem~\ref{thm4}, we have 
	$$\prescript{\perp}{}{\mathsf{U}^{-1}(\mathcal{Y})}=\Delta(\mathcal{X})=\mathsf{T}(\mathcal{X}).$$
	By Theorem~\ref{thm5},  $(\mathsf{T}(\mathcal{X}),\mathsf{U}^{-1}(\mathcal{Y}))$ is a hereditary complete cotorsion pair.
\end{proof}

When the abelian category $\mathcal{B}$ is Frobenius, its trivial extension also exhibits good properties. Also, applying Theorems~\ref{thm1} and ~\ref{thm2} to the cotorsion pair $(\mathcal{B},\;\prescript{}{\mathcal{B}}{\mathcal{I}})$, we have the following result.

\begin{thm}\label{thm6}
	Let $\mathcal{B}\ltimes\mathsf{F}$ be a right trivial extension of an abelian category $\mathcal{B}$ with $\mathsf{F}^2=0$.  Assume that $\mathsf{F}$ admits a right adjoint functor $\mathsf{G}$, that $\mathsf{F}$ is exact and that $\mathsf{F}$ preserves projective objects. If $\mathcal{B}$ is a Frobenius abelian category, then
	\begin{enumerate}[(1)]
		\item The category $\mathcal{B}\ltimes\mathsf{F}$ is a Gorenstein abelian category with $\mathrm{FPD}(\mathcal{B}\ltimes\mathsf{F})=\mathrm{FID}(\mathcal{B}\ltimes\mathsf{F})\le 1$, and $$\prescript{}{\mathcal{B}\ltimes\mathsf{F}}{\mathcal{P}}^{<\infty}=\prescript{}{\mathcal{B}\ltimes\mathsf{F}}{\mathcal{P}}^{\le 1}=\mathsf{U}^{-1}(\prescript{}{\mathcal{B}}{\mathcal{P}})=\mathsf{U}^{-1}(\prescript{}{\mathcal{B}}{\mathcal{I}})=\prescript{}{\mathcal{B}\ltimes\mathsf{F}}{\mathcal{I}}^{\le 1}=\prescript{}{\mathcal{B}\ltimes\mathsf{F}}{\mathcal{I}}^{<\infty}.$$
		\item $(\prescript{\perp}{}{\mathsf{U}^{-1}(\prescript{}{\mathcal{B}}{\mathcal{I}})},\mathsf{U}^{-1}(\prescript{}{\mathcal{B}}{\mathcal{I}}))=(\Delta(\mathcal{B}),\Delta(\mathcal{B})^\perp)=(\mathrm{GP}(\mathcal{B}\ltimes\mathsf{F}),\prescript{}{\mathcal{B}\ltimes\mathsf{F}}{\mathcal{P}}^{\le 1})$ are hereditary complete cotorsion pairs. So we have $\mathrm{GP}(\mathcal{B}\ltimes\mathsf{F})=\Delta(\mathcal{B})$.
	\end{enumerate}
\end{thm}

To prove Theorem~\ref{thm6}, we need the following lemmas.

\begin{lem}\label{proj}
	Let $\mathcal{B}\ltimes\mathsf{F}$ be a right trivial extension of an abelian category $\mathcal{B}$ with $\mathsf{F}^2=0$. Assume that $\mathsf{F}$ preserves projective objects. Let $(P,f)\in\mathcal{B}\ltimes\mathsf{F}$ with $P\in\prescript{}{\mathcal{B}}{\mathcal{P}}$. Then
	\begin{equation}\label{l1}0\longrightarrow(\mathsf{F}(P),0)\xlongrightarrow{\big(\begin{smallmatrix}
			-f\\ 1
		\end{smallmatrix}\big)}\mathsf{T}(P)\xlongrightarrow{(1,f)}(P,f)\longrightarrow0\end{equation}
	is a projective resolution of $(P,f)$. In particular, for any $(P,f)\in\mathcal{B}\ltimes\mathsf{F}$ with $P\in\prescript{}{\mathcal{B}}{\mathcal{P}}$, we have $\projdim(P,f)\le 1$.
\end{lem}
\begin{proof}
	We claim that $(\mathsf{F}(P),0)$ is a projective object of $\mathcal{B}\ltimes\mathsf{F}$. Since $\mathsf{F}$  preserves projective objects, $\mathsf{F}(P)$ is projective. Since $\mathsf{F}^2=0$, it follows from Proposition~\ref{pro1}~(3) that $(\mathsf{F}(P),0)\cong\mathsf{T}(\mathsf{F}(P))$ is projective. By Proposition~\ref{pro1}~(3), we obtain that $\mathsf{T}(P)\in\mathcal{B}\ltimes\mathsf{F}$ is projective.
	
	Since $\mathsf{F}^2=0$, $\big(\begin{smallmatrix}
		-f\\ 1
	\end{smallmatrix}\big)$ and $(1,f)$ are morphisms in $\mathcal{B}\ltimes\mathsf{F}$. Since
	\[0\longrightarrow\mathsf{F}(P)\xlongrightarrow{\big(\begin{smallmatrix}
			-f\\ 1
		\end{smallmatrix}\big)}P\oplus\mathsf{F}(P)\xlongrightarrow{(1,f)}P\longrightarrow0\]is exact in $\mathcal{B}$, the sequence (\ref{l1}) is exact.
\end{proof}

The following lemma is the dual of Lemma~\ref{proj}.

\begin{lem}\label{inj}
	Let $\mathsf{G}\rtimes\mathcal{B}$ be a left trivial extension of an abelian category $\mathcal{B}$ with $\mathsf{G}^2=0$. Assume that $\mathsf{G}$ preserves injective objects. Let $[I,g]\in\mathsf{G}\rtimes\mathcal{B}$ with $I\in\prescript{}{\mathcal{B}}{\mathcal{I}}$. Then
	\begin{equation}\label{l2}0\longrightarrow[I,g]\xlongrightarrow{\big(\begin{smallmatrix}
			g\\1
		\end{smallmatrix}\big)}\mathsf{H}(I)\xlongrightarrow{(1,-g)}[\mathsf{G}(I),0]\longrightarrow0\end{equation}
	is an injective resolution of $[I,g]$. In particular, for any $[I,g]\in\mathsf{G}\rtimes\mathcal{B}$ with $I\in\prescript{}{\mathcal{B}}{\mathcal{I}}$, we have $\injdim [I,g]\le 1$.
\end{lem}
\begin{proof}
	We claim that $[\mathsf{G}(I),0]$ is an injective object of $\mathsf{G}\rtimes\mathcal{B}$. Since $\mathsf{G}$ preserves injective objects, $\mathsf{G}(I)$ is injective. Since $\mathsf{G}^2=0$, it follows from Proposition~\ref{pro2}~(3) that $[\mathsf{G}(I),0]$ is injective. By Proposition~\ref{pro2}~(3), we obtain that $\mathsf{H}(I)\in\mathsf{G}\rtimes\mathcal{B}$ is injective.
	
	Since $\mathsf{G}^2=0$, $\big(\begin{smallmatrix}
		g\\ 1
	\end{smallmatrix}\big)$ and $(1,-g)$ are morphisms in $\mathsf{G}\rtimes\mathcal{B}$. Since
	\[0\longrightarrow I\xlongrightarrow{\big(\begin{smallmatrix}
			g\\ 1
		\end{smallmatrix}\big)}\mathsf{G}(I)\oplus I\xlongrightarrow{(1,-g)}\mathsf{G}(I)\longrightarrow0\]is exact in $\mathcal{B}$, the sequence (\ref{l2}) is exact.
\end{proof}

\begin{proof}[Proof of Theorem~\ref{thm6}]
	(1)	By Proposition~\ref{pro3}~(4), we have $\mathcal{B}\ltimes\mathsf{F}\cong\mathsf{G}\rtimes\mathcal{B}$. Since $\mathsf{F}$ is exact and preserves projective objects, $\mathsf{G}$ preserves injective objects. $\mathsf{F}^2=0$ implies $\mathsf{G}^2=0$.
	
	Let $(X,f)\in\prescript{}{\mathcal{B}\ltimes\mathsf{F}}{\mathcal{P}}$. By Proposition~\ref{pro1}~(3), we get $(X,f)\cong\mathsf{T}(P)$ with $P\in\prescript{}{\mathcal{B}}{P}$. Since $\mathsf{F}$ preserves projective objects and $\mathcal{B}$ is a Frobenius abelian category, $P\oplus\mathsf{F}(P)$ is injective. By Lemma~\ref{inj}, we have $\injdim \mathsf{T}(P)\le 1$.
	
 Let $[Y,g]\in\prescript{}{\mathsf{G}\rtimes\mathcal{B}}{\mathcal{I}}$. By Proposition~\ref{pro2}~(3), we get $[Y,g]\cong\mathsf{H}(I)$ with $I\in \prescript{}{\mathcal{B}}{I}$. Since $\mathsf{G}$ preserves injective objects and $\mathcal{B}$ is a Frobenius abelian category, $\mathsf{G}(I)\oplus I$ is projective. By Lemma~\ref{proj}, we have $\projdim \mathsf{H}(I)\le 1$. Therefore, we have $\mathrm{spli}(\mathcal{B}\ltimes\mathsf{F})\le 1$ and $\mathrm{silp}(\mathcal{B}\ltimes\mathsf{F})\le 1.$ Hence, $\mathcal{B}\ltimes\mathsf{F}$ is Gorenstein. By Theorem~\ref{Gor}~(1) and (2), we get $\prescript{}{\mathcal{B}\ltimes\mathsf{F}}{\mathcal{P}}^{<\infty}=\prescript{}{\mathcal{B}\ltimes\mathsf{F}}{\mathcal{P}}^{\le 1}=\prescript{}{\mathcal{B}\ltimes\mathsf{F}}{\mathcal{I}}^{\le 1}=\prescript{}{\mathcal{B}\ltimes\mathsf{F}}{\mathcal{I}}^{<\infty}$.
	
	 Since $\mathsf{F}$ preserves projective objects, it follows from Lemma~\ref{proj} that $\mathsf{U}^{-1}(\prescript{}{\mathcal{B}}{\mathcal{P}})\subseteq\prescript{}{\mathcal{B}\ltimes\mathsf{F}}{\mathcal{P}}^{\le 1}$. On the other hand, let $(X,f)\in\prescript{}{\mathcal{B}\ltimes\mathsf{F}}{\mathcal{P}}^{\le 1} $ and 
	\[0\longrightarrow\mathsf{T}(P_1)\longrightarrow\mathsf{T}(P_2)\longrightarrow (X,f)\longrightarrow0\]
	be a projective resolution of $(X,f)$. Here, $P_1,P_2\in\prescript{}{\mathcal{B}}{\mathcal{P}}$. Then we have an exact sequence
	\begin{equation}\label{l3}0\longrightarrow P_2\oplus\mathsf{F}(P_2)\longrightarrow P_1\oplus\mathsf{F}(P_1)\longrightarrow X\longrightarrow0.\end{equation}
	Since $\mathcal{B}$ is a Frobenius abelian category and $\mathsf{F}$ preserves projective objects, $P_2\oplus\mathsf{F}(P_2)$ is injective. Thus (\ref{l3}) splits and hence $X$ is projective. Therefore, $(X,f)\in\mathsf{U}^{-1}(\prescript{}{\mathcal{B}}{\mathcal{P}})$. Hence $\prescript{}{\mathcal{B}\ltimes\mathsf{F}}{\mathcal{P}}^{\le 1}=\mathsf{U}^{-1}(\prescript{}{\mathcal{B}}{\mathcal{P}})=\mathsf{U}^{-1}(\prescript{}{\mathcal{B}}{\mathcal{I}})$.
	
	Similarly, we have $\prescript{}{\mathcal{B}\ltimes\mathsf{F}}{\mathcal{I}}^{\le 1}=\mathsf{U}^{-1}(\prescript{}{\mathcal{B}}{\mathcal{I}})$. Hence, we have 
	$$\prescript{}{\mathcal{B}\ltimes\mathsf{F}}{\mathcal{P}}^{<\infty}=\mathsf{U}^{-1}(\prescript{}{\mathcal{B}}{\mathcal{P}})=\prescript{}{\mathcal{B}\ltimes\mathsf{F}}{\mathcal{P}}^{\le 1}=\mathsf{U}^{-1}(\prescript{}{\mathcal{B}}{\mathcal{I}})=\prescript{}{\mathcal{B}\ltimes\mathsf{F}}{\mathcal{I}}^{\le 1}=\prescript{}{\mathcal{B}\ltimes\mathsf{F}}{\mathcal{I}}^{<\infty}.$$
	
	(2) Since $\mathcal{B}\ltimes\mathsf{F}$ is Gorenstein and $\prescript{}{\mathcal{B}\ltimes\mathsf{F}}{\mathcal{P}}^{<\infty}=\mathsf{U}^{-1}(\prescript{}{\mathcal{B}}{\mathcal{I}})$, it follows from Theorem~\ref{Gor}~(3) that $\prescript{\perp}{}{\mathsf{U}^{-1}(\prescript{}{\mathcal{B}}{\mathcal{I}})}=\mathrm{GP}(\mathcal{B}\ltimes\mathsf{F})$. Therefore, $(\prescript{\perp}{}{\mathsf{U}^{-1}(\prescript{}{\mathcal{B}}{\mathcal{I}})},\mathsf{U}^{-1}(\prescript{}{\mathcal{B}}{\mathcal{I}}))=(\mathrm{GP}(\mathcal{B}\ltimes\mathsf{F}),\prescript{}{\mathcal{B}\ltimes\mathsf{F}}{\mathcal{P}}^{\le 1})$.
	
	Since $\mathsf{F}$ is exact, applying Proposition~\ref{pro3.10} to the cotorsion pair $(\mathcal{B},\prescript{}{\mathcal{B}}{\mathcal{I})}$, we get 
		$$
		\Delta(\mathcal{B})^\perp\subseteq \mathsf{U}^{-1}(\prescript{}{\mathcal{B}}{\mathcal{I}})=\prescript{}{\mathcal{B}\ltimes\mathsf{F}}{\mathcal{P}}^{\le 1}.
		$$ Thus, to prove $(\Delta(\mathcal{B}),\Delta(\mathcal{B})^\perp)=(\mathrm{GP}(\mathcal{B}\ltimes\mathsf{F}),\prescript{}{\mathcal{B}\ltimes\mathsf{F}}{\mathcal{P}}^{\le 1})$, it suffices to show $\Delta(\mathcal{B})\subseteq\mathrm{GP}(\mathcal{B}\ltimes\mathsf{F})=\prescript{\perp}{}{(\prescript{}{\mathcal{B}\ltimes\mathsf{F}}{\mathcal{P}}^{\le 1})}$.
		
		\textbf{Step 1.} We prove that $\Delta(\mathcal{B})\subseteq\prescript{\perp}{}{\mathsf{T}(\prescript{}{\mathcal{B}}{\mathcal{P}})}$. Let $(X,f)\in\Delta(\mathcal{B})$ and $P\in\prescript{}{\mathcal{B}}{\mathcal{P}}$. By Lemma~\ref{proj}, there is an exact sequence
		\[0\longrightarrow(\mathsf{F}(P),0)\xlongrightarrow{\big(\begin{smallmatrix}
				0\\ 1
			\end{smallmatrix}\big)}\mathsf{T}(P)\xlongrightarrow{(1,0)}(P,0)\longrightarrow0.\]
		By Lemma~\ref{Ext1}, we have
		\[\mathrm{Ext}_{\mathcal{B}\ltimes\mathsf{F}}^1((X,f),(P,0))\cong\mathrm{Ext}_\mathcal{B}^1(\mathrm{Coker}(f),P)=0\] and \[\mathrm{Ext}_{\mathcal{B}\ltimes\mathsf{F}}^1((X,f),(\mathsf{F}(P),0))\cong\mathrm{Ext}_\mathcal{B}^1(\mathrm{Coker}(f),\mathsf{F}(P))=0.\]
		Thus, we get $\mathrm{Ext}_{\mathcal{B}\ltimes\mathsf{F}}^1((X,f),\mathsf{T}(P))=0$. Hence, $\Delta(\mathcal{B})\subseteq\prescript{\perp}{}{\mathsf{T}(\prescript{}{\mathcal{B}}{\mathcal{P}})}$.
		
		\textbf{Step 2.} We prove that $\Delta(\mathcal{B})\subseteq\prescript{\perp}{}{(\prescript{}{\mathcal{B}\ltimes\mathsf{F}}{\mathcal{P}}^{\le 1})}.$  Let $(X,f)\in\Delta(\mathcal{B})$ and $(P,g)\in\mathsf{U}^{-1}(\prescript{}{\mathcal{B}}{\mathcal{P}})$. By Lemma~\ref{proj}, there is an exact sequence
		\begin{equation}\label{s0}0\longrightarrow(\mathsf{F}(P),0)\xlongrightarrow{\big(\begin{smallmatrix}
				-g\\ 1
			\end{smallmatrix}\big)}\mathsf{T}(P)\xlongrightarrow{(1,g)}(P,g)\longrightarrow0.\end{equation}
		Then there is an exact sequence
		\[0=\mathrm{Ext}_{\mathcal{B}\ltimes\mathsf{F}}^1((X,f),\mathsf{T}(P))\rightarrow\mathrm{Ext}_{\mathcal{B}\ltimes\mathsf{F}}^1((X,f),(P,g))\rightarrow\mathrm{Ext}_{\mathcal{B}\ltimes\mathsf{F}}^2((X,f),(\mathsf{F}(P),0)).\]
		Since $\projdim(\mathsf{F}(P),0)\le 1$ and $\mathrm{FID}(\mathcal{B}\ltimes\mathsf{F})\le 1$, we have $\mathrm{Ext}_{\mathcal{B}\ltimes\mathsf{F}}^2((X,f),(\mathsf{F}(P),0))=0.$ Therefore, $\mathrm{Ext}_{\mathcal{B}\ltimes\mathsf{F}}^1((X,f),(P,g))=0$. Hence, $\Delta(\mathcal{B})\subseteq\prescript{\perp}{}{(\prescript{}{\mathcal{B}\ltimes\mathsf{F}}{\mathcal{P}}^{\le 1})}.$  Then we are done.
\end{proof}

\subsection{Model structures in extensions of abelian categories}

In this subsection, we will  construct the Hovey triples and the $\omega$-model structures in the trivial extension categories. For Hovey triples and  $\omega$-model structures, we refer to Subsections~\ref{subsec2.6} and ~\ref{subsec2.7}.

\begin{thm}\label{thm9}
	Let $\mathcal{B}\ltimes\mathsf{F}$ be a right trivial extension of an abelian category $\mathcal{B}$ with $\mathsf{F}^2=0$. Assume that $\mathcal{M}_\mathcal{B}=(\mathcal{C,F,W})$ is a hereditary Hovey triple in $\mathcal{B}$. If $\mathbb{L}_1\mathsf{F}(\mathcal{C})=0$ and $\mathsf{F}(\mathcal{C})\subseteq\mathcal{F}\cap\mathcal{W}$, then 
	\begin{equation*}\label{triple1}
		\mathcal{M}_{\mathcal{B}\ltimes\mathsf{F}}=	\big(\mathsf{T}(\mathcal{C}),\;\mathsf{U}^{-1}(\mathcal{F}),\;\mathsf{U}^{-1}(\mathcal{W})\big)
	\end{equation*}
	is a hereditary Hovey triple in $\mathcal{B}\ltimes\mathsf{F}$ and there is a triangle equivalence $\mathrm{Ho}(\mathcal{M}_{\mathcal{B}\ltimes\mathsf{F}})\simeq\mathrm{Ho}(\mathcal{M}_\mathcal{B})$. 
\end{thm}

\begin{proof}
	 Since  $(\mathcal{C,F,W})$ is a hereditary Hovey triples in $\mathcal{B}$, $(\mathcal{C}\cap\mathcal{W},\mathcal{F})$ and $(\mathcal{C},\mathcal{F}\cap\mathcal{W})$ are hereditary complete cotorsion pairs.
	
 Since $\mathbb{L}_1\mathsf{F}(\mathcal{C}\cap\mathcal{W})\subseteq\mathbb{L}_1\mathsf{F}(\mathcal{C})=0$ and $\mathsf{F}(\mathcal{C}\cap\mathcal{W})\subseteq\mathcal{F}$, it follows from Corollary~\ref{cor1} that $(\mathsf{T}(\mathcal{C}\cap\mathcal{W}),\;\mathsf{U}^{-1}(\mathcal{F}))$ is a hereditary complete cotorsion pair.
	 
	 Similarly, $(\mathsf{T}(\mathcal{C}),\;\mathsf{U}^{-1}(\mathcal{F}\cap\mathcal{W}))$ is a hereditary complete cotorsion pair, by Theorem~\ref{thm5}.
	 
	 Since $\mathsf{F}(\mathcal{C})\subseteq\mathcal{W}\cap\mathcal{F}\subseteq\mathcal{W}$, we get $\mathsf{T}(\mathcal{C})\cap\mathsf{U}^{-1}(\mathcal{W})=\mathsf{T}(\mathcal{C}\cap\mathcal{W})$. Also, we have $\mathsf{U}^{-1}(\mathcal{F})\cap\mathsf{U}^{-1}(\mathcal{W})=\mathsf{U}^{-1}(\mathcal{F}\cap\mathcal{W})$. We mention that $\mathsf{U}^{-1}(\mathcal{W})$ is thick, since $\mathcal{W}$ is thick. Therefore, $\big(\mathsf{T}(\mathcal{C}),\;\mathsf{U}^{-1}(\mathcal{F}),\;\mathsf{U}^{-1}(\mathcal{W})\big)$ is a hereditary Hovey triple in $\mathcal{B}\ltimes\mathsf{F}$.
	 
	 Since $\mathsf{F}(\mathcal{C})\subseteq\mathcal{F}\cap\mathcal{W}$, it follows from Theorem~\ref{te} that we have the following triangle equivalence
	 \[\begin{aligned}
	 	\mathrm{Ho}(\mathcal{M}_{\mathcal{B}\ltimes\mathsf{F}})&\simeq(\mathsf{T}(\mathcal{C})\cap \mathsf{U}^{-1}(\mathcal{F}))/(\mathsf{T}(\mathcal{C})\cap\mathsf{U}^{-1}(\mathcal{F})\cap\mathsf{U}^{-1}(\mathcal{W}))\\
	 	&\simeq\mathsf{T}(\mathcal{C}\cap\mathcal{F})/\mathsf{T}(\mathcal{C}\cap\mathcal{F}\cap\mathcal{W})\\
	 	&\simeq (\mathcal{C}\cap\mathcal{F})/(\mathcal{C}\cap\mathcal{F}\cap\mathcal{W})\\
	 	&\simeq\mathrm{Ho}(\mathcal{M}_\mathcal{B}).
	 \end{aligned}\]
	 Here, the third equivalence is induced by $(X,f)\mapsto\mathrm{Coker}(f)$.
\end{proof} 
\begin{thm}\label{thm10}
	Let $\mathcal{B}\ltimes\mathsf{F}$ be a right trivial extension of an abelian category $\mathcal{B}$ with $\mathsf{F}^2=0$. Assume that $(\mathcal{X,Y})$ is a hereditary complete cotorsion pair with $\omega=\mathcal{X}\cap\mathcal{Y}$ contravariantly finite. If $\mathbb{L}_1(\mathcal{X})=0$ and $\mathsf{F}(\mathcal{X})\subseteq\mathcal{Y}$, then $(\mathsf{T}(\mathcal{X}),\mathsf{U}^{-1}(\mathcal{Y}))$ a hereditary complete cotorsion pair with $\omega'=\mathsf{T}(\mathcal{X})\cap \mathsf{U}^{-1}(\mathcal{Y})$ contravariantly finite. Moreover, there is an equivalence of categories $\mathrm{Ho}(\mathcal{M}_\omega)\simeq\mathrm{Ho}(\mathcal{M}_{\omega'}).$
\end{thm}
\begin{proof}
	By Corollary~\ref{cor1}, it suffices to show that $\omega'=\mathsf{T}(\mathcal{X})\cap \mathsf{U}^{-1}(\mathcal{Y})$ is contravariantly finite. Since $\mathsf{F}(\mathcal{X})\subseteq\mathcal{Y}$, we obtain $$\omega'=\mathsf{T}(\mathcal{X})\cap \mathsf{U}^{-1}(\mathcal{Y})=\mathsf{T}(\mathcal{X}\cap\mathcal{Y})=\mathsf{T}(\omega).$$
	For any $(A,\alpha)\in\mathcal{B}\ltimes\mathsf{F}$, since $\omega$ is contravariantly finite, there is a morphism $f_0:X_0\rightarrow A$ such that every  morphism $X\rightarrow A$ with $X\in\omega$ factors through $f_0$. Then $(f_0,\alpha\circ\mathsf{F}(f_0)):\mathsf{T}(X_0)\rightarrow (A,\alpha)$ is a morphism in $\mathcal{B}\ltimes\mathsf{F}$ with $\mathsf{T}(X_0)\in\omega'$. Let $(\beta_1,\beta_2):\mathsf{T}(X)\rightarrow (A,\alpha)$ be a morphism in $\mathcal{B}\ltimes\mathsf{F}$ with $X\in\omega$. Then $\beta_2=\alpha\circ\mathsf{F}(\beta_1)$. There is a morphism $\gamma:X\rightarrow X_0$ such that $\beta_1=f_0\circ\gamma$. Then $\mathsf{T}(\gamma):\mathsf{T}(X)\rightarrow\mathsf{T}(X_0)$ satisfies $$(f_0,\alpha\circ\mathsf{F}(f_0))\circ\mathsf{T}(\gamma)=(f_0\circ\gamma,\alpha\circ\mathsf{F}(f_0)\circ\mathsf{F}(\gamma))=(\beta_1,\beta_2).$$
	It follows that $\omega'=\mathsf{T}(\omega)$ is contravariantly finite, as desired.
	
	Finally, by Theorem~\ref{corr2}, we have
	\[\mathrm{Ho}(\mathcal{M}_{\omega'})\simeq\mathsf{T}(\mathcal{X})/\omega' =\mathsf{T}(\mathcal{X})/\mathsf{T}(\omega)\simeq\mathcal{X}/\omega \simeq\mathrm{Ho}(\mathcal{M}_{\omega'}).\]
	Here, the third equivalence is induced by $(X,f)\mapsto\mathrm{Coker}(f)$.
\end{proof}
\subsection{Remarks: the dual version}
​For convenience, we state the dual versions of the main results for $\zeta$-coextensions of abelian categories without proofs.​ For $\zeta$-coextensions of abelain categoies, we refer to Subsection~\ref{subsec2.2}.

 Let $\mathsf{G}\rtimes_\zeta\mathcal{B}$ be a $\zeta$-coextension of $\mathcal{B}$. For a class $\mathcal{X}$ of objects of $\mathcal{B}$, define
\[\mathsf{U}^{-1}(\mathcal{X}):=\left\{[X,f]\in\mathsf{G}\rtimes_\zeta\mathcal{B}\;|\;X\in\mathcal{X}\right\};\]
Denote by $\nabla(\mathcal{X})$ the class of objects $[X,f]$ in $\mathsf{G}\rtimes_\zeta\mathcal{B}$ such that the sequence $X\xrightarrow{f}\mathsf{G}(X)\xrightarrow{\mathsf{G}(f)-\zeta_X}\mathsf{G}^2(X)$ is exact and $\mathrm{Ker}(f)\in\mathcal{X}$. 
It is easy to check that $\mathsf{H}(\mathcal{X})\subseteq\nabla(\mathcal{X})$. It is clear that if $\mathsf{G}^2=0$, then 
\[\nabla(\mathcal{X})=\left\{[X,f]\in\mathsf{G}\rtimes\mathcal{B}\;|\;f \text{ is an epimorphism with }\mathrm{Ker}(f)\in\mathcal{X}\right\}.\]
For an additive functor $\mathsf{G}$ between abelian categories, let $\{\mathbb{R}^n\mathsf{G}\}_{n\in\mathbb{Z}}$ be the right derived functors of $\mathsf{G}$.

\begin{thm}\label{thm11}
	Let $\mathsf{G}\rtimes_\zeta\mathcal{B}$ be an $\zeta$-coextension of $\mathcal{B}$. Assume that $(\mathcal{X,Y})$ is a cotorsion pair in $\mathcal{B}$. If $\mathbb{R}^1\mathsf{G}(\mathcal{Y})=0$, then $(\mathsf{U}^{-1}(\mathcal{X}),\mathsf{U}^{-1}(\mathcal{X})^\perp)$ is a cotorsion pair in $\mathsf{G}\rtimes_\zeta\mathcal{B}$; and moreover, it is hereditary if and only if so is $(\mathcal{X,Y})$. 
\end{thm}

\begin{thm}\label{thm12}
	Let $\mathsf{G}\rtimes\mathcal{B}$ be a left trivial extension of $\mathcal{B}$ with $\mathsf{G}^2=0$. Assume that $(\mathcal{X,Y})$ is a cotorsion pair in $\mathcal{B}$. Then $(\prescript{\perp}{}{\nabla(\mathcal{Y})},\nabla(\mathcal{Y}))$ is a cotorsion pair in $\mathsf{G}\rtimes\mathcal{B}$. Moreover, if $\mathbb{R}^1\mathsf{G}(\mathcal{Y})=0$, then $(\prescript{\perp}{}{\nabla(\mathcal{Y})},\nabla(\mathcal{Y}))$ is hereditary if and only if so is $(\mathcal{X,Y})$.
\end{thm}

\begin{thm}\label{thm13}
	Let $\mathsf{G}\rtimes\mathcal{B}$ be a left trivial extension of an abelian category $\mathcal{B}$. Assume that $(\mathcal{X,Y})$ is a hereditary complete cotorsion pair, $\mathbb{R}^1\mathsf{G}(\mathcal{Y})=0$ and $\mathsf{G}(\mathcal{Y})\subseteq\mathcal{X}$. Then the following hold. 
	\begin{enumerate}[(1)]
	 \item $(\mathsf{U}^{-1}(\mathcal{X}),\mathsf{U}^{-1}(\mathcal{X})^\perp)$ is a hereditary complete cotorsion pair.
	 \item If $\mathsf{G}^2=0$, then $(\mathsf{U}^{-1}(\mathcal{X}),\nabla(\mathcal{Y}))$ is a hereditary complete cotorsion pair.
\end{enumerate}
\end{thm}

\section{Applications and examples}\label{sec4}

Some examples of extensions of abelian categories are summarised in \cite{B}. In this section, we study cotorsion pairs in the comma categories, some Morita context rings and trivial extensions of rings.

\subsection{Cotorsion pairs in comma categories}

Let $\mathcal{C}$ and $\mathcal{D}$ be abelian categories. For a right exact functor $\mathsf{G}:\mathcal{C}\rightarrow\mathcal{D}$, the \emph{comma category} \cite{M2}  $(\mathsf{G}\downarrow\mathcal{D})$ is defined as follows. The objects are $\big(\begin{smallmatrix}
	X\\ Y
\end{smallmatrix}\big)_f$ with $X\in\mathcal{C}$, $Y\in\mathcal{D}$ and $f\in\mathrm{Hom}_{\mathcal{D}}(\mathsf{G}X,Y)$. A morphism from $\big(\begin{smallmatrix}
X\\ Y
\end{smallmatrix}\big)_f$ to $\big(\begin{smallmatrix}
X'\\ Y'
\end{smallmatrix}\big)_{f'}$ is $\big(\begin{smallmatrix}\alpha\\ \beta\end{smallmatrix}\big)$ with $\alpha\in\mathrm{Hom}_\mathcal{C}(X,X')$ and $\beta\in\mathrm{Hom}_{\mathcal{D}}(Y,Y')$ such that $f'\circ \mathsf{G}(\alpha)=\beta\circ f.$

Let $\mathcal{B}=\mathcal{C}\times\mathcal{D}$ and $\mathsf{F}:\mathcal{B}\rightarrow\mathcal{B},\;(X,Y)\mapsto(0,\mathsf{G}X)$. Then $\mathsf{F}$ is right exact and $\mathsf{F}^2=0$. It is known that there is an isomorphism of categories $$\Phi:(\mathsf{G}\downarrow\mathcal{D})\xlongrightarrow{\sim}\mathcal{B}\ltimes\mathsf{F},\;\big(\begin{smallmatrix}
	X\\ Y
\end{smallmatrix}\big)_f\longmapsto ((X,Y),(0,f)).$$

For a class $\mathcal{X}$ of objects of $\mathcal{C}$ and a class $\mathcal{Y}$ of objects of $\mathcal{D}$, let
\begin{align*}\left(\begin{smallmatrix}\mathcal X\\ \mathcal Y\end{smallmatrix}\right): & =  \ \{\left(\begin{smallmatrix} X \\ Y\end{smallmatrix}\right)_{f}\in (\mathsf{G}\downarrow\mathcal{D}) \ \mid \   X\in \mathcal X, \ \ Y\in\mathcal Y\};\\[7pt]
	\mathrm{R}(\mathcal X, \ \mathcal Y): & = \ \{\left(\begin{smallmatrix} X \\ Y\end{smallmatrix}\right)_{f}\in (\mathsf{G}\downarrow\mathcal{D}) \ \mid X\in\mathcal{X},\;\mathrm{Coker}(f)\in\mathcal{Y} \\ &\;\;\;\;\;\;\;\;\;\;\;\;\;\;\;\;\;\;\;\;\;\;\;\text{and }f\text{ is a monomorphism}\}.
\end{align*}

We mention that $(\mathcal{X,Y})$ is a class of objects of $\mathcal{B}$. We recall $\mathsf{U}^{-1}((\mathcal{X},\;\mathcal{Y}))$ and $\Delta((\mathcal{X},\;\mathcal{Y}))$ in Subsection~\ref{3.1}. Then we have
\[\Phi\left(\begin{smallmatrix}\mathcal X\\ \mathcal Y\end{smallmatrix}\right)=\mathsf{U}^{-1}((\mathcal{X},\;\mathcal{Y}))\;\mbox{ and }\;\Phi(\mathrm{R}(\mathcal{X},\;\mathcal{Y}))=\Delta((\mathcal{X},\;\mathcal{Y})).\]

There are some characterizations of cotorsion pairs and special precovering classes in comma categories in \cite{HW,HZ}.

\begin{thm}\label{thm7}
	Keep the notation as above. Let $(\mathcal{U},\;\mathcal{X})$ and $(\mathcal{V},\;\mathcal{Y})$ be cotorsion pairs in $\mathcal{C}$ and $\mathcal{D}$, respectively. 
	\begin{enumerate}[(1)]
		\item If $\mathbb{L}_1\mathsf{G}(\mathcal{U})=0$, then $\big(\prescript{\perp}{}{\left(\begin{smallmatrix}
				\mathcal{X}\\ \mathcal{Y}
			\end{smallmatrix}\right)},\;\left(\begin{smallmatrix}
			\mathcal{X}\\ \mathcal{Y}
		\end{smallmatrix}\right)\big)$ is a cotorsion pair in $(\mathsf{G}\downarrow\mathcal{D})$; and moreover, it is hereditary if and only if so are $(\mathcal{U}\;\mathcal{X})$ and $(\mathcal{V},\;\mathcal{Y})$.
		\item $(\mathrm{R}(\mathcal U, \mathcal V),\;\mathrm{R}(\mathcal U, \mathcal V)^\perp)$ is a cotorsion pair in $(\mathsf{G}\downarrow\mathcal{D})$; and moreover, if $\mathbb{L}_1\mathsf{G}(\mathcal{X})=0$, then it is hereditary if and only if so are $(\mathcal{U},\;\mathcal{X})$ and $(\mathcal{V},\;\mathcal{Y})$.
		\item If $\mathbb{L}_1\mathsf{G}(\mathcal{U})=0$, then we have $\big(\prescript{\perp}{}{\left(\begin{smallmatrix}
				\mathcal{X}\\ \mathcal{Y}
			\end{smallmatrix}\right)},\;\left(\begin{smallmatrix}
			\mathcal{X}\\ \mathcal{Y}
		\end{smallmatrix}\right)\big)=(\mathrm{R}(\mathcal U,\mathcal V),\;\mathrm{R}(\mathcal U, \mathcal V)^\perp).$
	\end{enumerate}
\end{thm}
\begin{proof}
	We write $\mathcal{L}=(\mathcal{U,V})$ and $\mathcal{R}=(\mathcal{X,Y})$. 
	
	Since $(\mathcal{U},\;\mathcal{X})$ and $(\mathcal{V},\;\mathcal{Y})$ is cotorsion pairs in $\mathcal{C}$ and $\mathcal{D}$, $(\mathcal{L},\;\mathcal{R})$ is a cotorsion pair in $\mathcal{B}$. We mention that $\mathbb{L}_1\mathsf{F}(\mathcal{L})=(0,\mathbb{L}_1\mathsf{G}(\mathcal{U}))$ and $\mathbb{L}_1\mathsf{F}(\mathcal{R})=(0,\mathbb{L}_1\mathsf{G}(\mathcal{X}))$. Applying Theorem~\ref{thm1} and Theorem~\ref{thm2}, we get (1) and (2), respectively.
	
	(3) We prove that $\mathsf{U}^{-1}(\mathcal{R})=\Delta(\mathcal{L})^\perp.$ Since $\mathbb{L}_1\mathsf{F}(\mathcal{L})=(0,\mathbb{L}_1\mathsf{G}(\mathcal{U}))=0$, it follows from Proposition~\ref{pro3.10} that $\Delta(\mathcal{L})^\perp\subseteq\mathsf{U}^{-1}(\mathcal{R})$.
	
	To show the opposite inclusion, let $((X,Y),(0,f))\in\mathsf{U}^{-1}(\mathcal{R})$. There is an exact sequence
	\[0\longrightarrow((0,Y),(0,0))\xlongrightarrow{(0,\mathsf{Id}_Y)}((X,Y),(0,f))\xlongrightarrow{(\mathsf{Id}_X,0)}((X,0),(0,0))\longrightarrow0\]
	 in $\mathcal{B}$. For any $L\in\Delta(\mathcal{L})$, by Lemma~\ref{Ext1}, we have
	 \[\mathrm{Ext}_{\mathcal{B}\ltimes\mathsf{F}}^1(L,\mathsf{Z}(0,Y))\cong\mathrm{Ext}_\mathcal{B}^1(\mathsf{C}(L),(0,Y))=0,\]
	 since $\mathsf{C}(L)\in\mathcal{L}$ and $(0,Y)\in\mathcal{R}$. Thus, $\mathsf{Z}(0,Y)\in\Delta(\mathcal{L})^\perp$. Similarly, we have $\mathsf{Z}(X,0)\in\Delta(\mathcal{L})^\perp$. Since $\Delta(\mathcal{L})^\perp$ is closed under extensions. Therefore, we get $((X,Y),(0,f))\in\Delta(\mathcal{L})^\perp$. Hence, $\Delta(\mathcal{L})^\perp\supseteq\mathsf{U}^{-1}(\mathcal{R})$. Then we are done.
\end{proof}

\begin{thm}\label{thm8}
	Keep the notation as above. Let $(\mathcal{U},\;\mathcal{X})$ and $(\mathcal{V},\;\mathcal{Y})$ be hereditary complete cotorsion pairs in $\mathcal{C}$ and $\mathcal{D}$, respectively. If $\mathbb{L}_1\mathsf{G}(\mathcal{U})=0$ and $\mathsf{G}(\mathcal{U})\subseteq\mathcal{Y}$, then the following hold.
	\begin{enumerate}[(1)]
		\item $\big(\prescript{\perp}{}{\left(\begin{smallmatrix}
				\mathcal{X}\\ \mathcal{Y}
			\end{smallmatrix}\right)},\;\left(\begin{smallmatrix}
			\mathcal{X}\\ \mathcal{Y}
		\end{smallmatrix}\right)\big)$ is a hereditary complete cotorsion pair in $(\mathsf{G}\downarrow\mathcal{D})$.
		\item We have
		\[\prescript{\perp}{}{\left(\begin{smallmatrix}
				\mathcal{X}\\ \mathcal{Y}
			\end{smallmatrix}\right)}=\left\{\left(\begin{smallmatrix}
			U\\ \mathsf{G}(U)
		\end{smallmatrix}\right)_{1}\oplus\left(\begin{smallmatrix}
			0\\ V
		\end{smallmatrix}\right)_{0}\;\mid\;U\in\mathcal{U},\;V\in\mathcal{V}\right\}.\]
		\item If $\omega_1=\mathcal{U}\cap\mathcal{X}$ and $\omega_2=\mathcal{V}\cap\mathcal{Y}$ are contravariantly finite, then $$\omega=\prescript{\perp}{}{\left(\begin{smallmatrix}
				\mathcal{X}\\ \mathcal{Y}
			\end{smallmatrix}\right)}\cap\left(\begin{smallmatrix}
			\mathcal{X}\\ \mathcal{Y}
		\end{smallmatrix}\right)=\left\{\left(\begin{smallmatrix}
		U\\ \mathsf{G}(U)
		\end{smallmatrix}\right)_{1}\oplus\left(\begin{smallmatrix}
		0\\ V
		\end{smallmatrix}\right)_{0}\;\mid\;U\in\mathcal{U}\cap\mathcal{X},\;V\in\mathcal{V}\cap\mathcal{Y}\right\}$$
		is contravariantly finite, and there is an equivalence of categories $\mathrm{Ho}(\mathcal{M}_\omega)\simeq\mathrm{Ho}(\omega_1)\times\mathrm{Ho}(\omega_2)$.
	\end{enumerate}
\end{thm}

\begin{proof}
	We write $\mathcal{L}=(\mathcal{U,V})$ and $\mathcal{R}=(\mathcal{X,Y})$. Then $(\mathcal{L},\mathcal{R})$ is a hereditary complete cotorsion pair in $\mathcal{B}$.
	
	 Since $\mathbb{L}_1\mathsf{F}(\mathcal{L})=(0,\mathbb{L}_1\mathsf{G}(\mathcal{U}))=0$ and $\mathsf{F}(\mathcal{L})=(0,\mathsf{G}(\mathcal{U}))\subseteq\mathcal{R}$, it follows from Theorem~\ref{thm4} that $\prescript{\perp}{}{\mathsf{U}^{-1}(\mathcal{R})}=\mathsf{T}(\mathcal{L})$. Therefore, we have
	\[\prescript{\perp}{}{\left(\begin{smallmatrix}
			\mathcal{X}\\ \mathcal{Y}
		\end{smallmatrix}\right)}=\left\{\left(\begin{smallmatrix}
		U\\ \mathsf{G}(U)
	\end{smallmatrix}\right)_{1}\oplus\left(\begin{smallmatrix}
		0\\ V
	\end{smallmatrix}\right)_{0}\;\mid\;U\in\mathcal{U},\;V\in\mathcal{V}\right\}.\]
	By Corollary~\ref{cor1}, $(\mathsf{T}(\mathcal{L}),\mathsf{U}^{-1}(\mathcal{R}))$ is a hereditary complete cotorsion pair in $\mathcal{B}\ltimes\mathsf{F}$. Therefore, $\big(\prescript{\perp}{}{\left(\begin{smallmatrix}
			\mathcal{X}\\ \mathcal{Y}
		\end{smallmatrix}\right)},\;\left(\begin{smallmatrix}
		\mathcal{X}\\ \mathcal{Y}
	\end{smallmatrix}\right)\big)$ is a hereditary complete cotorsion pair in $(\mathsf{G}\downarrow\mathcal{D})$. For (3), it follows from Theorem~\ref{thm10}.
\end{proof}

\begin{thm}
	Keep the notation as above. Let $\mathcal{M}_\mathcal{C}=(\mathcal{U},\mathcal{X},\mathcal{W}_1)$ and $\mathcal{M}_\mathcal{D}=(\mathcal{V},\mathcal{Y},\mathcal{W}_2)$ be hereditary  Hovey triples in $\mathcal{C}$ and $\mathcal{D}$, respectively. If $\mathbb{L}_1\mathsf{G}(\mathcal{U})=0$ and $\mathsf{G}(\mathcal{U})\subseteq\mathcal{Y}\cap\mathcal{W}_2$, then 
	$$
	\mathcal{M}=(\mathcal{U}',\;\left(\begin{smallmatrix}
		\mathcal{X}\\ \mathcal{Y}
	\end{smallmatrix}\right),\;\left(\begin{smallmatrix}
	\mathcal{W}_1\\ \mathcal{W}_2
	\end{smallmatrix}\right))
	$$ is a hereditary Hovey triple in $(\mathsf{G}\downarrow\mathcal{D})$ and there is a triangle equivalence $\mathrm{Ho}(\mathcal{M})\simeq\mathrm{Ho}(\mathcal{M}_\mathcal{C})\times\mathrm{Ho}(\mathcal{M}_\mathcal{D})$. Here, $\mathcal{U}'=\left\{\big(\begin{smallmatrix}
		U\\\mathsf{G}(U)
	\end{smallmatrix}\big)_{1}\oplus\left(\begin{smallmatrix}
	0\\ V
	\end{smallmatrix}\right)_{0}\;\mid\;U\in\mathcal{U},\;V\in\mathcal{V}\right\}$.
\end{thm}
\begin{proof}
	Let $\mathcal{L}=(\mathcal{U,V})$, $\mathcal{F}=(\mathcal{X,Y})$ and $\mathcal{W}=(\mathcal{W}_1,\mathcal{W}_2)$. Then $(\mathcal{L},\;\mathcal{F},\;\mathcal{W})$ is a hereditary complete Hovey triple in $\mathcal{B}$. The results follow from Theorem~\ref{thm9}.
\end{proof}

\subsection{Examples}
\begin{exm}[Morita context rings]

Let $A$ and $B$ be rings, $M$ a $B$-$A$-bimodule and $N$ an $A$-$B$-bimodule. A Morita context ring \cite{Bass} with zero bimodule homomorphisms is $\Lambda:=\begin{pmatrix}
	A&N\\ M&B
\end{pmatrix}$ with componentwise addition, and multiplication
\[\begin{pmatrix}
	a_1&n_1\\
	m_1&b_1
\end{pmatrix}\begin{pmatrix}
a_2&n_2\\ m_2&b_2
\end{pmatrix}=\begin{pmatrix}
a_1a_2&a_1n_2+n_1b_2\\ m_1a_2+b_1m_2&b_1b_2
\end{pmatrix}.\]
\end{exm}

Let $\mathcal{B}=(A\text{-}\mathrm{Mod})\times(B\text{-}\mathrm{Mod})$, $\mathsf{F}:\mathcal{B}\rightarrow\mathcal{B},\;(X,Y)\mapsto(N\otimes_B Y,\;M\otimes_A X)$, and $\mathsf{G}:\mathcal{B}\rightarrow\mathcal{B},\;(X,Y)\mapsto(\mathrm{Hom}_B(M,Y),\;\mathrm{Hom}_A(N,X))$. Then there is the following expressions of module categories of Morita context rings; see \cite[Theorem~1.5]{Gr}.

\begin{pro}
	There are isomorphisms of categories $\Lambda$-$\mathrm{Mod}\cong\mathcal{B}\ltimes\mathsf{F}\cong\mathsf{G}\rtimes\mathcal{B}$.
\end{pro}

For an object $((X,Y),(f,g))\in\mathcal{B}\ltimes\mathsf{F}$, we write $\big(\begin{smallmatrix}
	X\\ Y
\end{smallmatrix}\big)_{f,g}$ as a $\Lambda$-module.

The cotorsion pairs and model structures on Morita context rings have been studied in \cite{CRZ}. By applying Theorems~\ref{thm1} and~\ref{thm11}, we obtain \cite[Theorem~1.1]{CRZ}. We mention that $\mathbb{L}_i\mathsf{F}=(\mathrm{Tor}_i^A(N,-),\;\mathrm{Tor}_i^B(M,-))$ and that $\mathsf{F}^2=0$ if and only if $M\otimes _AN=0=N\otimes_BM$.  Then Theorems~\ref{thm2} and~\ref{thm12}  recover \cite[Theorem~1.2]{CRZ}.  Applying Theorem~\ref{thm5} and Corollary~\ref{cor1}, we have the following result; compare \cite[Theorem~1.5]{CRZ}.

\begin{thm}
	Let $\Lambda=\begin{pmatrix}
		A&N\\ M&B
	\end{pmatrix}$ be a Morita context ring. Let $(\mathcal{U,X})$ and $(\mathcal{V,Y})$ be hereditary complete cotorsion pairs in $A$-$\mathrm{Mod}$ and $B$-$\mathrm{Mod}$, respectively. Assume that $\mathrm{Tor}_1^A(M,\mathcal{U})=0=\mathrm{Tor}_1^B(N,\mathcal{V})$. If $M\otimes_A \mathcal{U}\subseteq\mathcal{Y}$ and $N\otimes_B\mathcal{V}\subseteq\mathcal{X}$, then the following hold.
	\begin{enumerate}[(1)]
		\item $\big(\prescript{\perp}{}{\left(\begin{smallmatrix}\mathcal X\\ \mathcal Y\end{smallmatrix}\right)},\left(\begin{smallmatrix}\mathcal X\\ \mathcal Y\end{smallmatrix}\right)\big)$ is a hereditary complete cotorsion pair.
		\item If  $M\otimes_AN=0=N\otimes_BM$, then $(\mathsf{T}_A(\mathcal{U})\oplus\mathsf{T}_B(\mathcal{V}),\big(\begin{smallmatrix}
			\mathcal{X}\\ \mathcal{Y}
		\end{smallmatrix}\big))$ is a hereditary complete cotorsion pair in $\Lambda$-$\mathrm{Mod}$.
	\end{enumerate}
	Here, the functor $\mathsf{T}_A: A\text{-}\mathrm{Mod}\rightarrow\Lambda\text{-}\mathrm{Mod}$ is given by $X\mapsto\big(\begin{smallmatrix}
	X\\ M\otimes_A X
	\end{smallmatrix}\big)_{0,1}$, the functor $\mathsf{T}_B:B\text{-}\mathrm{Mod}\rightarrow\Lambda\text{-}\mathrm{Mod}$ is given by $Y\mapsto\big(\begin{smallmatrix}
	N\otimes_BY\\ Y
	\end{smallmatrix}\big)_{1,0}$, and 
	$$
	\left(\begin{smallmatrix}\mathcal X\\ \mathcal Y\end{smallmatrix}\right)=  \ \{\left(\begin{smallmatrix} X \\ Y\end{smallmatrix}\right)_{f,g}\in\Lambda\text{-}\mathrm{Mod} \ \mid \   X\in \mathcal X, \ \ Y\in\mathcal Y\}.
	$$
\end{thm}

Similarly, applying Theorems~\ref{thm9} and~\ref{thm10}, one has the corresponding results; compare \cite[Theorem~1.6]{CRZ}.

\begin{exm}[Trivial extensions of rings]
\emph{	Let $R$ be a ring and $M$ an $R$-$R$-bimodule. The \emph{trivial extension} of $R$ by $M$, denoted by $R\ltimes M$, is the ring $R\times M$ with componentwise addition, and  multiplication given by $(r_1,m_1)(r_2,m_2)=(r_1r_2,r_1m_2+m_1r_2)$. Let $\mathcal{B}=R\text{-}\mathrm{Mod}$, $\mathsf{F}=M\otimes_R-$ and $\mathsf{G}=\mathrm{Hom}_R(M,-)$. Then the category $(R\ltimes M)\text{-}\mathrm{Mod}$ is isomorphic to $\mathcal{B}\ltimes\mathsf{F}$ and is isomorphic to $\mathsf{G}\rtimes\mathcal{B}$; see \cite[Proposition~1.13]{FPR}. We mention that Theorems~\ref{thm1} and~\ref{thm11} recover \cite[Theorem~4.2~(1) and ~(2)]{Mao2}. It is of interest to compare Theorem~\ref{thm5} and \cite[Theorem~5.3]{Mao2}.}
\end{exm}

	\vskip 10pt
	
\noindent {\bf Acknowledgements.}\quad The author thanks his supervisor Professor Xiao-Wu Chen for his guidance, thanks Professor Jiangsheng Hu for many helpful suggestions and thanks Taolue Long many helpful comments.

	\vskip 10pt
	{\footnotesize \noindent Dongdong Hu\\
		School of Mathematical Sciences, University of Science and Technology of China, Hefei 230026, Anhui, PR China}
	
\end{CJK}
\end{document}